\documentclass[11pt]{article}

\usepackage{amsmath,amsfonts,amsthm,amscd,amssymb,graphicx}

\numberwithin{equation}{section}

\newtheorem{theorem}{Theorem}[section]

\newtheorem{lemma}[theorem]{Lemma}

\newtheorem{proposition}[theorem]{Proposition}
\newtheorem{prop}[theorem]{Proposition}

\newtheorem{remark}[theorem]{Remark}
\newtheorem{definition}[theorem]{Definition}

\def\eps{\varepsilon }
\def\D{\partial }

\newcommand{\RR}{\mathbb{R}}
\newcommand{\cO}{\mathcal{O}}

\newcommand{\CC}{\mathbb{C}}

\newcommand{\ZZ}{{\mathbb Z}}

\def\beq{\begin{equation}}
\def\eeq{\end{equation}}
\def\bb1{{1\!\!1}}
\def\R{\mbox{Re }}
\def\I{\mbox{Im }}

\def\dz{\partial_z}
\def\dx{\partial_x}

\def\rit{{\Bbb R}}
\def\cit{{\Bbb C}}

\def\eps{\varepsilon}

\begin{document}

\title{Spectral instability of characteristic boundary layer flows
}

\author{Emmanuel Grenier\footnotemark[1]
 \and Yan Guo\footnotemark[2] \and Toan T. Nguyen\footnotemark[3]
}

\date\today

\maketitle

\begin{abstract}

In this paper, we construct growing modes of the linearized Navier-Stokes equations about generic stationary shear flows of the boundary layer type in a regime of sufficiently large Reynolds number: $R \to \infty$. Notably, the shear profiles are allowed to be linearly stable at the infinite Reynolds number limit, and so the instability presented is purely due to the presence of viscosity. The formal construction of approximate modes is well-documented in physics literature, going back to the work of Heisenberg, C.C. Lin, Tollmien, Drazin and Reid, but a rigorous construction requires delicate mathematical details, involving for instance a treatment of primitive Airy functions and singular solutions. Our analysis gives exact unstable eigenvalues and eigenfunctions, showing that the solution could grow slowly at the rate of $e^{t/\sqrt {R}}$. A new, operator-based approach is introduced, avoiding to deal with matching inner and outer asymptotic expansions, but instead involving a careful study of singularity in the critical layers by deriving pointwise bounds on the Green function of the corresponding Rayleigh and Airy operators.



\end{abstract}

\renewcommand{\thefootnote}{\fnsymbol{footnote}}

\footnotetext[1]{Equipe Projet Inria NUMED,
 INRIA Rh\^one Alpes, Unit\'e de Math\'ematiques Pures et Appliqu\'ees., 
 UMR 5669, CNRS et \'Ecole Normale Sup\'erieure de Lyon,
               46, all\'ee d'Italie, 69364 Lyon Cedex 07, France. Email: egrenier@umpa.ens-lyon.fr}

\footnotetext[2]{Division of Applied Mathematics, Brown University, 182 George street, Providence, RI 02912, USA. Email: Yan\underline{~}Guo@Brown.edu}

\footnotetext[3]{Department of Mathematics, Penn State University, State College, PA 16803. Email: nguyen@math.psu.edu. TN's research is supported in part by the NSF under grant DMS-1338643.}

\tableofcontents

\newpage

\section{Introduction}

Study of hydrodynamics stability and the inviscid limit  of viscous fluids is one of the most classical subjects in fluid dynamics, going back to the most prominent physicists including Lord Rayleigh, Orr, Sommerfeld, Heisenberg, among many others. It is documented in the physical literature (see, for instance, \cite{LinBook,Reid}) that laminar viscous fluids are unstable, or become turbulent, in a small viscosity or high Reynolds number limit. In particular, generic stationary shear flows are linearly unstable for sufficiently large Reynolds numbers. In the present work and in another concurrent work of ours \cite{GGN-channel}, we provide a complete mathematical proof of these physical results.

Specifically, let $u_0= (U(z),0)^{tr}$ be a stationary shear flow. We are interested in the linearization of the incompressible Navier-Stokes equations about the shear profile: 
\begin{subequations}
\begin{align}
v_t +   u_0 \cdot \nabla v + v \cdot \nabla u_0  + \nabla p &= \frac {1}{R} \Delta v  \label{NS1}
\\
\nabla \cdot v &= 0 \label{NS2}
\end{align}
\end{subequations}
posed on $\RR\times \RR_+$, together with the classical no-slip boundary conditions on the walls:
\begin{equation}\label{NS3}
v_{\vert_{z=0}} = 0.  
\end{equation}
Here $v$ denotes the usual velocity perturbation of the fluid, and $p$ denotes the corresponding pressure. Of interest is the Reynolds number $R$ sufficiently large, and whether the linearized problem is spectrally unstable: the existence of unstable modes of the form $(v,p) =  (e^{\lambda t} \tilde v(y,z), e^{\lambda t} \tilde  p(y,z))$ for some $\lambda$ with $\Re \lambda >0$.


The spectral problem is a very classical issue in fluid mechanics.
A huge literature is devoted to its detailed study. We in particular refer to
\cite{Reid, Schlichting} for the major works of Heisenberg, C.C. Lin, Tollmien, and Schlichting.
The studies began around 1930, motivated by the study of the boundary layer
around wings. In airplanes design, it is crucial to study the boundary layer
around the wing, and more precisely the transition between the laminar and turbulent
regimes, and even more crucial to predict the point where boundary layer
splits from the boundary. A large number of papers has been devoted to 
the estimation of the critical Rayleigh number of classical shear flows 
(plane Poiseuille flow, Blasius profile, exponential suction/blowing profile, among others).

It were Sommerfeld and Orr \cite{Sommerfeld, Orr} who initiated the study of the spectral problem via the  Fourier normal mode theory. They search for the unstable solutions of the form $e^{i\alpha (y-ct)} (\hat v(z), \hat p(z))$, and derive the well-known Orr-Somerfeld equations for linearized viscous fluids: 
\beq \label{OS1-intro}
\epsilon (\dz^2 - \alpha^2)^2 \phi 
= (U-c) ( \dz^2 - \alpha^2) \phi  - U'' \phi, 
\eeq
with $\epsilon = 1/(i\alpha R)$, where $\phi(z)$ denotes the corresponding stream function, with $\phi$ and $\partial_z \phi$ vanishing at the boundary $z = 0$. When $\epsilon = 0$, \eqref{OS1-intro} reduces to the classical Rayleigh equation, which corresponds to inviscid flows. The singular perturbation theory was developed to construct Orr-Somerfeld solutions from those of Rayleigh solutions. 

\bigskip

{\bf Inviscid unstable profiles.} If the profile is unstable for the Rayleigh equation, then there exist a spatial frequency $\alpha_\infty$, an eigenvalue $c_\infty$ with
$\I c_\infty > 0$, and a corresponding eigenvalue $\phi _\infty$ that solve \eqref{OS1-intro} with $\epsilon =0$ or $R = \infty$. 
We can then make a perturbative analysis to construct an unstable eigenmode $\phi _R$ of
the Orr-Sommerfeld equation with an eigenvalue $\I c_R >0$ for any large enough $R$. This can be done by adding a boundary sublayer to the inviscid mode $\phi_\infty$ to correct the boundary conditions for the viscous problem. In fact, we can further check that 
\beq \label{perturb1}
c_R = c_\infty + \cO(R^{-1}) ,
\eeq
as $R\to \infty$. Thus, the time growth is of order $e^{\theta_0 t}$, for some $\theta_0>0$. Such a perturbative argument for the inviscid unstable profiles is well-known; see, for instance, Grenier \cite{Gr1} where he rigorously establishes the nonlinear instability of inviscid unstable profiles.


\bigskip

{\bf Inviscid stable profiles.} There are various criteria to check whether a shear profile is stable to the Rayleigh equation. The most classical one was due to Rayleigh \cite{Ray}: {\em A necessary condition for instability is that $U(z)$ must have an inflection point}, or its refined version by Fjortoft \cite{Reid}:  {\em A necessary condition for instability is that $U'' (U - U(z_0))<0$ somewhere in the flow, where $z_0$ is a point at which $U''(z_0) =0$.} For instance, the classical Blasius boundary layer profile is linearly stable to the Rayleigh equation.

For such a stable profile, all the spectrum of the Rayleigh equation
is imbedded on the imaginary axis: $\R(-i\alpha c_\infty) = \alpha \I c_\infty = 0$, and thus it is not clear whether a perturbative argument to construct solutions $(c_R,\phi_R)$ to \eqref{OS1-intro} would yield stability ($\I c_R <0$) or instability ($\I c_R >0$). It is documented in the physical literature that {\em generic shear profiles (including those which are inviscid stable) are linearly unstable for large Reynolds numbers.} Heisenberg \cite{Hei,HeiICM}, then Tollmien and C. C. Lin \cite{LinBook} were among the first physicists to use asymptotic expansions to study the instability; see also Drazin and Reid \cite{Reid} for a complete account of the physical literature on the subject. Roughly speaking, there are lower and upper marginal stability branches $\alpha_\mathrm{low}(R), \alpha_\mathrm{up}(R)$ so that whenever $\alpha\in [\alpha_\mathrm{low}(R),\alpha_\mathrm{up}(R)]$, there exist an unstable eigenvalue $c_R$ and an eigenfunction $\phi_R(z)$ to the Orr-Sommerfeld problem. The asymptotic behavior of these branches $\alpha_\mathrm{low}$ and $\alpha_\mathrm{up}$ depends on the profile:

\begin{itemize}

\item for plane Poiseuille flow in a channel: $U(z) = 1- z^2 $ for $-1 < z < 1$, 
\begin{equation}\label{ranges-alpha0}\alpha_\mathrm{low}(R) = A_{1c}R^{-1/7}\qquad  \mbox{and}\qquad  \alpha_\mathrm{up}(R)= A_{2c} R^{-1/11}.\end{equation}

\item for boundary layer profiles,
\begin{equation}\label{ranges-alphabl}\alpha_\mathrm{low}(R) = A_{1c}R^{-1/4}\qquad  \mbox{and}\qquad  \alpha_\mathrm{up}(R)= A_{2c} R^{-1/6}.\end{equation}   

\item for Blasius (a particular boundary layer) profile,
\begin{equation}\label{ranges-alpha}\alpha_\mathrm{low}(R) = A_{1c}R^{-1/4}\qquad  \mbox{and}\qquad  \alpha_\mathrm{up}(R)= A_{2c} R^{-1/10}.\end{equation} 

\end{itemize}

Their formal analysis has been compared with modern numerical computations
and also with experiments, showing a very good agreement; see \cite[Figure 5.5]{Reid} for a sketch of the marginal stability curves. In this paper, we are interested in the case of boundary layers.

%
%


In his works \cite{W1,W3, Wbook}, Wasow developed the turning point theory to rigorously validate the formal asymptotic expansions used by the physicists in a full neighborhood of the turning points (or the critical layers in our present paper). It appears however that Wasow himself did not explicitly study how his approximate solutions depend on the three small parameters $\alpha, \epsilon,$ and $\I c$ in the Orr-Sommerfeld equations, nor apply his theory to resolve the stability problem (see his discussions on pages 868--870, \cite{W1}, or Chapter One, \cite{Wbook}). 

 Even though Drazin and Reid (\cite{Reid}) indeed provide many delicate asymptotic analysis in different regimes with different matching conditions near the critical layers, it is mathematically unclear how to combine their ``local'' analysis into a single convergent ``global expansion'' to produce an exact growing mode for the Orr-Sommerfeld equation. To our knowledge, remarkably, after all these efforts, a complete rigorous construction of an unstable growing mode is still elusive for such a fundamental problem.





Our main result is as follows. 

\begin{theorem}\label{theo-unstablemodes}
Let $U(z)$ be an arbitrary shear profile with $U'(0) > 0$ and satisfy 
$$
\sup_{z \ge 0} | \partial^k_z (U(z) - U_+) e^{\eta_0 z} | < + \infty, \qquad k=0,\cdots ,4,
$$ for some constants $U_+$ and $\eta_0 > 0$. Let $\alpha_\mathrm{low}(R)$ and $\alpha_\mathrm{up}(R)$ be defined as in \eqref{ranges-alphabl} for general boundary layer profiles, or defined as in \eqref{ranges-alpha} for the Blasius profiles: those with additional assumptions: $U''(0) = U'''(0) = 0$. 

Then, there is a critical Reynolds number $R_c$ so that for all $R\ge R_c$ and all $\alpha \in (\alpha_\mathrm{low}(R), \alpha_\mathrm{up}(R))$, there exist
a nontrivial triple $c(R), \hat v(z; R), \hat p(z;R)$, with $\mathrm{Im} ~c(R) >0$, such that $v_R: = e^{i\alpha(y-ct) }\hat v(z;R)$ and $p_R: = e^{i\alpha(y-ct)} \hat p(z;R)$ solve the problem 
\eqref{NS1}-\eqref{NS2} with the no-slip boundary conditions. In the case of instability, there holds the following estimate for the growth rate of the unstable solutions:
$$ \alpha \I c(R) \quad \approx\quad  R^{-1/2},$$
as $R \to \infty$. 
\end{theorem}

Theorem \ref{theo-unstablemodes} allows general shear profiles.
 The instability is found, even for inviscid stable flows such as monotone or Blasius boundary flows,
  and thus is due to the presence of viscosity. For a fixed viscosity, nonlinear instability follows from the spectral instability; see \cite{FPS} for arbitrary spectrally unstable steady states. However, in the vanishing viscosity limit, linear to nonlinear instability is a very delicate issue, primarily due to the fact that there are no available, comparable bounds on the linearized solution operator as compared to the maximal growing mode. Available analyses (for instance, \cite{Fri, Gr1}) do not appear applicable in the inviscid limit.





%

As mentioned earlier, we construct the unstable solutions via the Fourier normal mode method. Precisely, let us introduce the stream function $\psi$ through
\begin{equation}\label{def-stream}v = \nabla^\perp \psi = (\partial_z, -\partial_y)\psi 
,\qquad \psi(t,y,z) := \phi (z) e^{i \alpha (y - ct) },
\end{equation}
with $y\in \RR$, $z\in \RR_+$, the spatial frequency $\alpha \in \RR$ and the temporal eigenvalue $c\in \CC$. The equation for vorticity $\omega = \Delta \psi$ becomes the classical Orr--Sommerfeld equation for $\phi$
\beq \label{OS1}
\epsilon (\dz^2 - \alpha^2)^2 \phi 
= (U-c) ( \dz^2 - \alpha^2) \phi  - U'' \phi  ,\qquad z\ge 0, \eeq
with $\epsilon = {1 \over i \alpha R}.$ The no-slip boundary condition on $v$ then becomes 
\beq \label{OS2}
\alpha \phi  = \partial_z \phi  = 0 \quad\hbox{ at } \quad z = 0,
\eeq
In addition, as we work with Sobolev spaces, we also impose the zero boundary conditions at infinity: \beq \label{OS3}
\phi \to 0 \quad \hbox{ and } \quadÊ\partial_z \phi  \to 0\quad  \hbox{ as } \quad z \to + \infty.
\eeq
Clearly, if $\phi(z)$ solves the Orr-Sommerfeld problem  \eqref{OS1}-\eqref{OS3}, then the velocity $v$ defined as in \eqref{def-stream} solves the linearized Navier-Stokes problem with the pressure $p$ solving 
$$ -\Delta p = \nabla U \cdot \nabla v, \qquad \partial_z p _{\vert_{z=0,2}} = -\partial_z^2\partial_y \psi_{\vert_{z=0,2}}. $$
Throughout the paper, we study the Orr-Sommerfeld problem.

Delicacy in the construction is primarily due to the formation of  critical layers. To see this, let $(c_0, \phi_0)$ be a solution to the Rayleigh problem with $c_0\in \RR$. Let $z_0$ be the point at which 
\begin{equation}\label{cr-layer} U(z_0) = c_0.\end{equation}
Since the coefficient of the highest-order derivative in the Rayleigh equation vanishes at $z = z_0$, the Rayleigh solution $\phi_0(z)$ has a singularity of the form: $1+(z-z_0) \log(z-z_0)$. A perturbation analysis to construct an Orr-Sommerfeld solution $\phi_\epsilon$ out of $\phi_0$ will face a singular source $\epsilon (\partial_z^2 - \alpha^2)^2 \phi _0$ at $z = z_0$. To deal with the singularity, we need to introduce the critical layer $\phi_\mathrm{cr}$ that solves 
$$\epsilon \partial_z^4 \phi_\mathrm{cr} = (U-c)\partial_z^2 \phi_\mathrm{cr}$$
When $z$ is near $z_0$, $U - c$ is approximately $z-z_c$ with $z_c$ near $z_0$, and the above equation for the critical layer becomes the classical Airy equation for $\partial_z^2 \phi_\mathrm{cr}$. This shows that the critical layer mainly depends on the fast variable: $\phi_\mathrm{cr} = \phi_\mathrm{cr}(Y)$ with $Y = (z-z_c)/\epsilon^{1/3}$. 
 
In the literature, the point $z_c$ is occasionally referred to as a turning point, since the eigenvalues of the associated first-order ODE system cross at $z=z_c$ (or more precisely, at those which satisfy $U(z_c) = c$), and therefore it is delicate to construct asymptotic solutions that are analytic across different regions near the turning point. In his work, Wasow fixed the turning point to be zero, and were able to construct asymptotic solutions in a full neighborhood of the turning point. 

   
In the present paper, we introduce a new, operator-based approach, which avoids dealing with inner and outer asymptotic expansions, but instead constructs the Green's function, and therefore the inverse, of the corresponding Rayleigh and Airy operators. The Green's function of the critical layer (Airy) equation is complicated by the fact that we have to deal with the second primitive Airy functions, not to mention that the argument $Y$ is complex. The basic principle of our construction, for instance, of a slow decaying solution, will be as follows. We start with an exact Rayleigh solution $\phi_0$ (solving \eqref{OS1} with $\epsilon =0$). This solution then solves \eqref{OS1} approximately up to the error term $\epsilon (\partial_z^2 - \alpha^2)^2 \phi _0$, which is singular at $z=z_0$ since $\phi_0$ is of the form $1+ (z-z_0)\log(z-z_0)$ inside the critical layer. We then correct $\phi_0$ by adding a critical layer profile $\phi_\mathrm{cr}$ constructed by convoluting the Green's function of the primitive Airy operator against the singular error $\epsilon (\partial_z^2 - \alpha^2)^2 \phi _0$. The resulting solution $\phi_0 + \phi_\mathrm{cr}$ solves \eqref{OS1} up to a smaller error that consists of no singularity. An exact slow mode of \eqref{OS1} is then constructed by inductively continuing this process. For a fast mode, we start the induction with a second primitive Airy function.

~\\
{\bf Notation.} 
Throughout the paper, the profile $U = U(z)$ is kept fixed. Let $c_0$ and $z_0$ be real numbers so that $U(z_0) = c_0$. We extend $U(z)$ analytically in a neighborhood of $z_0$ in $\CC$. We then let $c$ and $z_c$ be two complex numbers in the neighborhood of $(c_0,z_0)$ in $\mathbb{C}^2$ so that $U(z_c) = c$. It follows by the analytic expansions of $U(z)$ near $z_0$ and $z_c$ that $|\I c| \approx |\I z_c | $, provided that $U'(z_0)  \not =0$. Without loss of generality, we have taken $z_0 = 0$ in the statement of the theorem. 

~\\
{\bf Further notation.} We shall use $C_0$ to denote a universal constant that may change from line to line, but is independent of $\alpha$ and $R$. We also use the notation $f=\cO(g)$ or $f\lesssim g$ to mean that $|f|\le C_0 |g|$, for some constant $C_0$. Similarly, $f \approx g$ if and only if $f \lesssim g$ and $g \lesssim f$. Finally, when no confusion is possible, inequalities involved with complex numbers $|f| \le g$ are understood as $|f|\le |g|$.


\section{Strategy of proof}


Let us outline the strategy of the proof before going into the technical
details and computations. Our ultimate goal is to construct four independent solutions of the fourth order differential equation (\ref{OS1})
and then combine them in order to satisfy boundary conditions (\ref{OS2})-\eqref{OS3}, yielding the linear dispersion relation. The unstable eigenvalues are then found by carefully studying the dispersion relation.


\subsection{Operators}


For our convenience, let us introduce the following operators. We denote by  $Orr$ the Orr-Sommerfeld operator
\beq \label{opOrr}
Orr(\phi) := (U - c) (\dz^2 - \alpha^2) \phi - U'' \phi - \eps (\dz^2 - \alpha^2)^2 \phi,
\eeq
by $Ray_\alpha$ the Rayleigh operator
\beq \label{opRay}
Ray_\alpha(\phi): = (U-c) (\dz^2 - \alpha^2) \phi - U'' \phi,
\eeq
by $Diff$ the diffusive part of the Orr-Sommerfeld operator,
\beq \label{opDiff}
Diff(\phi) :=  - \eps (\dz^2 - \alpha^2)^2 \phi ,
\eeq
by $Airy$ the modified Airy equation
\beq \label{opAiry}
Airy(\phi) := \eps \dz^4 \phi - (U - c + 2 \eps \alpha^2) \dz^2 \phi ,
\eeq
and finally, by $Reg$ the regular zeroth order part of the Orr-Sommerfeld operator
\beq \label{opReg}
Reg(\phi): =- \Big[\eps \alpha^4 + U'' + \alpha^2 (U-c) \Big]\phi .
\eeq
Clearly, there hold identities
 \begin{equation}\label{key-ids}
Orr = Ray_\alpha + Diff = - Airy + Reg.
\end{equation}


\subsection{Asymptotic behavior as $z \to + \infty$}


In order to construct the independent solutions of \eqref{OS1}, let us study their possible behavior at infinity. One observes that as $z \to +\infty$, solutions of (\ref{OS1}) must behave like solutions of constant-coefficient limiting equation:
\beq \label{OS1inf}
\eps \partial_z^4 \phi = (U_+ - c + 2 \eps \alpha^2) \partial_z^2 \phi - \alpha^2 (\eps \alpha^2 + U_+ - c)\phi,
\eeq
with $U_+ = U(+\infty)$. Solutions to \eqref{OS1inf} are of the form $C e^{\lambda z}$ with 
$\lambda = \pm \lambda_s$ or $\lambda = \pm \lambda_f$, where
$$
\lambda_s = \pm \alpha + \cO(\alpha^2 \sqrt \eps ), \qquad \lambda_f = \pm \frac {1}{\sqrt \eps} (U_+-c)^{1/2} + \cO(\alpha).
$$
Therefore, we can find two solutions $\phi_1,\phi_2$ with a ``slow behavior" $\lambda \approx \pm \alpha$
(one decaying and the other growing) and two solutions $\phi_3,\phi_4$  with a fast behavior
where $\lambda$ is of order $\pm 1 / \sqrt{\eps}$ (one decaying and the other growing).
As it will be clear from the proof, the first two slow-behavior solutions $\phi_1$ and $\phi_2$ will
be perturbations of eigenfunctions of the Rayleigh equation. The other two, $\phi_3$ and
$\phi_4$, are specific to the Orr Sommerfeld equation and will be linked to the solutions of the classical Airy equation.
More precisely, four independent solutions of (\ref{OS1}) to be constructed are

\begin{itemize}

\item $\phi_1$ and $\phi_2$ which are perturbations of the decreasing/increasing eigenvector of 
the Rayleigh equation. To leading order in small $\alpha$ and $\eps$, $\phi_1$ and $\phi_2$ behave at infinity, respectively, like
$(U(z) - c) \exp(-\alpha z)$ and $(U(z) - c) \exp(\alpha z)$.


\item $\phi_3$ and $\phi_4$ which are perturbations of the solutions to the second primitive Airy equation, which are of order $\exp({\pm |Z|^{3/2}})$ as $|Z|\to \infty$. Here $Z = \eta(z)/\epsilon^{1/3}$ denotes the fast variable near the critical layer whose size is of order $\epsilon^{1/3}$, and $\eta(z)$ denotes the Langer's variable which is asymptotically $z^{2/3}$ as $z\to \infty$. 
 
\end{itemize}
A solution to the problem  (\ref{OS1})--(\ref{OS3}) is defined as a linear combination of $\phi_1, \phi_2, \phi_3,$ and $\phi_4$, solving the imposed boundary conditions. Keeping in mind the asymptotic behavior of
$\phi_2$ and $\phi_4$, we observe that any bounded solution of (\ref{OS1})--(\ref{OS3})
is in fact just a combination of $\phi_1$ and $\phi_3$. We will therefore restrict our construction
to the study of $\phi_1$ and $\phi_3$.

\subsection{Outline of the construction}

We now present the idea of the iterative construction. 
We start from the Rayleigh solution $\phi_{Ray}$ so that 
$$
Ray_\alpha (\phi_{Ray}) = f.
$$
By definition, we have
\begin{equation}\label{Orr-1stapp} Orr(\phi_{Ray}) = f - Diff (\phi_{Ray}).\end{equation}
Here we observe that the error term on the right hand side $Diff(\phi_{Ray}) =  \epsilon (\dz^2 - \alpha^2)^2 \phi_{Ray}$, denoted by $O_1(z)$, is of order $\cO(\epsilon)$ in $L^\infty$. It might be helpful to note that the operator $\dz^2 - \alpha^2$ and so $Diff(\cdot)$ annihilate the slow decay term $ \cO(e^{-\alpha z})$ in $\phi_{Ray}$.  Near the critical layer, the Rayleigh solution generally contains a singular solution of the form $(z-z_c)\log (z-z_c)$, 
 and therefore $\phi_{Ray}$ admits the same singularity at $z = z_c$. 
 As a consequence, $Diff (\phi_{Ray})$ consists of singularities of orders $log(z-z_c)$ 
 and $(z-z_c)^{-k}$, for $k=1,2,3$. To remove these singularities,
we then use the $Airy$ operator. More precisely, 
the $Airy$ operator  smoothes out the singularity inside the critical layer. In term of spatial decaying at infinity, the inverse of the $Airy(\cdot)$ operator introduces some linear growth in the spatial variable, which prevents the convergence of our iteration. We then introduce yet another modified Airy operator $\mathcal{A}_\mathrm{a}(\cdot)$ so that  
$$Airy(\phi) = \mathcal{A}_\mathrm{a} (\partial_z^2 \phi).$$

We then proceed our contruction by defining 
\begin{equation}\label{def-phi1-intro}
\phi_{1} := \phi_{Ray} + Airy^{-1} (A_s) + \partial_z^{-2}\mathcal{A}^{-1}_{\mathrm{a}} (I_0) ,
\end{equation}
in which $A_{s}= \chi  Diff (\phi_{Ray})$ denoting the singular part, $I_0 = (1-\chi ) Diff (\phi_{Ray})$ denoting the regular part, and  $\partial_z^{-1} = -\int_z^\infty$. Here, $\chi(z)$ is a smooth cut-off function such that $\chi = 1$ on $[0,1]$ and zero on $[2,\infty)$. We then get 
 $$
Orr( \phi_{1} ) =  f + O_1, \qquad O_1:=  Reg \Big( Airy^{-1} (A_s) + \partial_z^{-2}\mathcal{A}^{-1}_{\mathrm{a}} (I_0) \Big) .
$$
Our main technical task is to show that $O_1$ is indeed in the next vanishing order, when $\epsilon \to 0$, or precisely the iteration operator 
\begin{equation}\label{def-Iter-intro}
\begin{aligned}
Iter :&=  Reg \circ \Big[ Airy^{-1} \circ \chi Diff   + \partial_z^{-2}\mathcal{A}^{-1}_{\mathrm{a}}  \circ (1-\chi )Diff  \Big] \circ Ray_\alpha^{-1}
 \end{aligned}
\end{equation}
is contractive in suitable function spaces. Note that our approach avoids to deal with inner and outer expansions, but requires a careful
study of the singularities and delicate estimates on the resolvent solutions.


\subsection{Function spaces}


Throughout the paper, $z_c$ is some complex number and will be fixed, depending only on $c$, through $U(z_c) = c$. 

We will use the function spaces $X_{p}^\eta$, for $p\ge 0$, to denote the spaces consisting of measurable functions $f = f(z)$ such that the norm 
$$
\| f\|_{X_p^\eta} : = \sup_{|z-z_c|\le 1 } \sum_{k=0}^p|  (z-z_c)^k \dz^k f(z) | + \sup_{|z-z_c|\ge 1} \sum_{k=0}^p|  e^{\eta z} \dz^k f(z) |  
$$ 
is bounded. In case $p=0$, we simply write $X_\eta, \|\cdot \|_\eta$ in places of $X^{\eta}_0, \|\cdot \|_{X_0^\eta}$, respectively. 

We also introduce the function spaces $Y_p^\eta \subset X_p^\eta$, $p\ge 0$, such that for any $f\in Y_p^\eta$, the function $f$ additionally satisfies 
$$
|f(z)| \le C, \qquad | \dz f(z) | \le C (1 + | \log (z - z_c) | ) , \qquad  | \dz^k f(z) | \le C (1 + | z - z_c |^{1 - k} )
$$
for all $|z-z_c|\le 1$ and for $2\le k \le p$. The best constant $C$ in the previous bounds defines the norm $\| f \|_{Y_p^\eta}$.

Let us now sketch the key estimates of the paper. The first point is, thanks to almost
explicit computations, we can construct
an inverse operator $Ray^{-1}$ for $Ray_\alpha$. Note that if $Ray_{\alpha} (\phi) = f$,
then
\beq \label{Rayl}
(\partial_z^2 - \alpha^2) \phi = {U'' \over U - c} \phi + {f \over U - c} .
\eeq
Hence, provided $U-c$ does not vanish (which is the case when $c$ is complex), using classical
elliptic regularity we see that if $f \in C^k$ then $\phi \in C^{k+2}$. We thus gain two derivatives.
However the estimates on the derivatives degrade as $z - z_c$ goes smaller. The main point is
that the weight $(z-z_c)^l$ is enough to control this singularity. Moreover, deriving $l$ times 
(\ref{Rayl}) we see that $\partial_z^{2+l} \phi$ is bounded by $C / (z - z_c)^{l+1}$ if $f \in X_{\eta,k}$. Hence,
we gain one $z - z_c$ factor in the derivative estimates between $f$ and $\phi$. In addition, since $e^{\pm \alpha z} 
$ is in the kernel of $\partial_z^2 -\alpha^2$, if $f$ decays like $e^{-\eta z}$, one can at best expect $\phi$ to decay as $e^{-\alpha z}$ at infinity. 
Combining, if $f$ lies in $X^\eta_{k}$, $\phi$ lies in $Y^\alpha_{k+2}$, with a gain of two derivatives and of an extra
$z - z_c$ weight, but losses a rapid decay at infinity. As a matter of fact we will construct an inverse $Ray^{-1}$ 
 which is continuous from $X^\eta_{k}$ to $Y^\alpha_{ k+2}$ for any $k$. 

\medskip

Using Airy functions, their double primitives, and a special variable and unknown transformation 
known in the literature as Langer transformation, we can construct an almost explicit inverse $Airy^{-1}$
to our $Airy$ operator.
We then have to investigate $Airy^{-1} \circ Diff$. Formally it is of order $0$, however it is singular,
hence to control it we need to use two derivatives, and to make it small we need a $z - z_c$ factor
in the norms. After tedious computations on almost explicit Green functions
 we prove that $Airy^{-1} \circ Diff$
has a small norm as an operator from
$Y^\alpha_{ k+2}$ to $X^\eta_{k}$. 

\medskip

Last, $Reg$ is bounded from $X^\eta_{k}$ to $X^\eta_{k}$, since it is a simple multiplication by a bounded function.
Combining all these estimates we are able to construct exact solutions of Orr Sommerfeld equations,
starting from solutions of Rayleigh equations of from Airy equations. This leads to the construction of
four independent solutions. Each such solution is defined as a convergent serie, which gives
its expansion. It then remains to combine all the various terms of all these solutions to get the dispersion
relation of Orr Sommerfeld. The careful analysis of this dispersion relation gives our instability result.

The plan of the paper follows the previous lines.


\newpage
\section{Rayleigh equation}\label{sec-Rayleigh}


In this part, we shall construct an exact inverse for the Rayleigh operator $Ray_\alpha$
for small $\alpha$ and so find the complete solution to 
\begin{equation}\label{eq-Raya}
Ray_\alpha(\phi) = (U-c)(\dz^2 - \alpha^2)\phi - U'' \phi = f
\end{equation}
To do so, we first invert the Rayleigh operator $Ray_0$ when $\alpha =0$ by exhibiting an explicit Green function. 
We then use this inverse to construct an approximate inverse to $Ray_\alpha$ operator through the construction 
of an approximate Green function. Finally, the construction of the exact inverse of $Ray_\alpha$ follows by  an iterative procedure. 

Precisely, we will prove in this section the following proposition. 

\begin{proposition}\label{prop-exactRayS} Let $p$ be in $\{0,1,2\}$ and $\eta>0$. Assume that $\I c \not =0$ and $\alpha |\log \I c|$ is sufficiently small. Then, there exists an operator $RaySolver_{\alpha,\infty} (\cdot) $ from $X_p^\eta$ to $Y^\alpha_{p+2}$ (defined by \eqref{def-exactRayS}) so that 
\begin{equation}\label{eqs-RaySolver}
\begin{aligned}
 Ray_\alpha (RaySolver_{\alpha,\infty} (f)) &= f.
\end{aligned} \end{equation} 
In addition, there holds 
$$\| RaySolver_{\alpha,\infty}(f)\|_{Y^\alpha_{p+2}} \le C \|f\|_{X_p^\eta}(1+|\log (\I c)|) ,$$
for all $f \in X_p^\eta$.
\end{proposition}


\subsection{Case $\alpha = 0$}

As mentioned, we begin with the Rayleigh operator $Ray_0$ when
$\alpha = 0$. We will find the inverse of $Ray_0$. More precisely, we will construct the Green function of $Ray_0$ and solve  
\begin{equation}\label{Ray0} Ray_0 (\phi) = (U-c) \dz^2 \phi - U'' \phi = f.\end{equation}


We recall that $z_c$ is defined by solving the equation $U(z_c) = c$. We first prove the following lemma.

\begin{lemma}\label{lem-defphi012} Assume that $\I c \not =0$. There are two independent solutions $\phi_{1,0},\phi_{2,0}$ of $Ray_0(\phi) =0$ with the Wronskian determinant 
$$ 
W(\phi_{1,0}, \phi_{2,0}) := \dz \phi_{2,0} \phi_{1,0} - \phi_{2,0} \dz \phi_{1,0} = 1.
$$
Furthermore, there are analytic functions $P_1(z), P_2(z), Q(z)$ with $P_1(z_c) = P_2(z_c) = 1$ and $Q(z_c)\not=0$ so that the asymptotic descriptions 
\begin{equation}\label{asy-phi012} 
\phi_{1,0}(z) = (z-z_c) P_1(z) ,\qquad \phi_{2,0}(z) = P_2(z) + Q(z) (z-z_c) \log (z-z_c)
\end{equation}
hold for $z$ near $z_c$, and  \begin{equation}\label{decay-phi012} 
| \phi_{1,0}(z) - V_+| \le Ce^{-\eta_0 |z|} , \qquad |\dz \phi_{2,0}(z)  - \frac{1}{V_+}  |\le Cze^{-\eta_0|z|},
\end{equation}
as $|z|\to \infty$, for some positive constants $C,\eta_0$ and for $V_+ = U_+ - c$. Here when $z-z_c$ is on the negative real axis, we take the value of $\log (z-z_c)$ to be $ \log |z-z_c| - i \pi$.   
 \end{lemma}
\begin{proof} First, we observe that $$ \phi_{1,0}(z) = U(z)-c$$ is an exact solution of $Ray_0(\phi) =0$. In addition, the claimed asymptotic expansion for $\phi_{1,0}$ clearly holds for $z$ near $z_c$ since $U(z_c) = c$. We then construct a second particular solution $\phi_{2,0}$, imposing the Wronskian determinant to be one:
$$
W[\phi_{1,0},\phi_{2,0}] =  \dz \phi_{2,0} \phi_{1,0} - \phi_{2,0} \dz \phi_{1,0}  = 1.
$$
From this, the variation-of-constant method $\phi_{2,0} (z)= C(z) \phi_{1,0}(z)$ then yields
$$
  \phi_{1,0} C \partial_z  \phi_{1,0} 
+ \ \phi_{1,0}^2 \partial_z C -\partial_z  \phi_{1,0} C  \phi_{1,0}  = 1.
$$
This gives $
\partial_z C(z) = 1/ \phi^2_{1,0}(z)$ and therefore
\beq \label{defiphi2}
\phi_{2,0}(z) = (U(z) - c)  \int_{1/2}^z {1 \over (U(y) - c)^2} dy.
\eeq
Note that $\phi_{2,0}$ is well defined if the denominator does not vanishes, hence if
$\I c \not =  0$ or if $\I c = 0$ and $z>  z_c$. 
More precisely,
$$
{1 \over (U(c) - U(z_c))^2}
= {1 \over ( U'(z_c) (z - z_c) + U''(z_c) (z - z_c)^2 / 2 + ...)^2}
$$
$$
= {1 \over U'(z_c)^2 (z-z_c)^2} 
- {U''(z_c) \over U'(z_c)^3} {1 \over z - z_c} + holomorphic .
$$
Hence
\beq \label{defiphi2bis}
\phi_{2,0} = - {U(z) - c \over U'(z_c)^2 (z - z_c)}
- {U''(z_c) \over U'(z_c)^3} (U(z) -  c)\log (z - z_c)   + holomorphic .
\eeq
As $\phi_{2,0}$ is not properly defined for $z < z_c$ when $z_c \in \rit^+$, it is coherent
to choose the determination of the logarithm which is defined on $\cit - \rit^-$.

With such a choice of the logarithm, $\phi_{2,0}$ is holomorphic in $\cit - \{ z_c + \rit^-  \}$.
In particular if $\Im z_c = 0$, $\phi_{2,0}$ is holomorphic in $z$ excepted on the half line $z_c + \rit^- $.
For $z \in \rit$, $\phi_{2,0}$ is holomorphic as a function of $c$ excepted if $z - z_c$ is real and negative,
namely excepted if $z < z_c$. 
For a fixed $z$, $\phi_{2,0}$ is an holomorphic function of $c$ provided $z_c$ does not cross
$\rit^+$, and provided $z - z_c$ does not cross $\rit^-$.
The Lemma then follows from the explicit expression (\ref{defiphi2bis}) of $\phi_{2,0}$.
\end{proof}

Let $\phi_{1,0},\phi_{2,0}$ be constructed as in Lemma \ref{lem-defphi012}. Then the Green function $G_{R,0}(x,z)$ of the $Ray_0$ operator can be defined by 
$$
G_{R,0}(x,z) = \left\{ \begin{array}{rrr} (U(x)-c)^{-1} \phi_{1,0}(z) \phi_{2,0}(x), 
\quad \mbox{if}\quad z>x,\\
(U(x)-c)^{-1} \phi_{1,0}(x) \phi_{2,0}(z), \quad \mbox{if}\quad z<x.\end{array}\right.
$$ 
Here we note that $c$ is complex with $\I c \not=0$ and so the Green function $G_{R,0}(x,z)$ is a well-defined function in $(x,z)$, continuous across $x=z$, and its first derivative has a jump across $x=z$. Let us now introduce the inverse of $Ray_0$ as 
\begin{equation}\label{def-RayS0}
\begin{aligned}
RaySolver_0(f) (z)  &: =  \int_0^{+\infty} G_{R,0}(x,z) f(x) dx.
\end{aligned}
\end{equation}

The following lemma asserts that the operator $RaySolver_0(\cdot)$ is in fact well-defined from $X_0^\eta$ to $Y^0_2$, which in particular shows that $RaySolver_0(\cdot)$ gains two derivatives, but losses the fast decay at infinity.  

%


%

\begin{lemma}\label{lem-RayS0} Assume that $\I c \not =0$. For any $f\in {X_0^\eta}$,  the function $RaySolver_0(f)$ is a solution to the Rayleigh problem \eqref{Ray0}. In addition, $RaySolver_0(f) \in Y^0_2$, and there holds  
$$
\| RaySolver_0(f)\|_{Y^0_2} \le C (1+|\log \I c|) \|f\|_{{X_0^\eta}},
$$ 
for some universal constant $C$. 
\end{lemma}
\begin{proof} As long as it is well-defined, the function $RaySolver_0(f)(z)$ solves the equation \eqref{Ray0} at once by a direct calculation, upon noting that 
$$ Ray_0 (G_{R,0}(x,z) ) = \delta_x(z),$$
for each fixed $x$. 

Next, by scaling, we assume that $ \| f\|_{X_0^\eta} = 1$. By Lemma \ref{lem-defphi012}, it is clear that $\phi_{1,0}(z)$ and $\phi_{2,0}(z)/(1+z)$
  are uniformly bounded. Thus, by direct computations, we have 
\begin{equation}\label{est-Gr0}|G_{R,0}(x,z)| \le  C \max\{ (1+x), |x-z_c|^{-1} \}.\end{equation}
That is, $G_{R,0}(x,z)$ grows linearly in $x$ for large $x$ and has a singularity of order $|x-z_c|^{-1}$ when $x$ is near $z_c$, for arbitrary $z \ge 0$.  Since $|f(z)|\le e^{-\eta z}$, the integral \eqref{def-RayS0} is well-defined and satisfies 
$$|RaySolver_0(f) (z)| \le  C \int_0^\infty e^{-\eta x} \max\{ (1+x), |x-z_c|^{-1} \}  \; dx\le C (1+|\log \I c|),$$
in which we used the fact that $\I z_c \approx \I c$. 
  
Finally, as for derivatives, we need to check the order of singularities for $z$ near $z_c$. We note that $|\partial_z \phi_{2,0}| \le C (1+|\log(z-z_c)|)$, and hence 
  $$|\partial_zG_{R,0}(x,z)| \le  C \max\{ (1+x), |x-z_c|^{-1} \} (1+|\log(z-z_c)|).$$
Thus, $\partial_z RaySolver_0(f)(z)$ behaves as $1+|\log(z-z_c)|$ near the critical layer. In addition, from the $Ray_0$ equation, we have \begin{equation}\label{identity-R0f} \dz^2 (RaySolver_0(f)) = \frac{U''}{U-c} RaySolver_0(f) + \frac{f}{U-c}.\end{equation}
This proves that $RaySolver_0(f) \in Y_2^0$ by definition of the function space $Y_2^0$. \end{proof}

\begin{lemma}\label{lem-derRayS0}Assume that $\I c \not =0$. Let $p$ be in $\{0,1,2\}$. For any $f \in X_p^\eta$, we have 
$$
\begin{aligned}
 \| RaySolver_0(f) \|_{Y_{p+2}^0} \le C \|f\|_{X_p^\eta}(1+|\log (\I c)| )
 \end{aligned}$$ 
\end{lemma}
\begin{proof} This is Lemma \ref{lem-RayS0} when $p=0$. When $p=1$ or $2$, the lemma follows directly from the identity \eqref{identity-R0f}.
\end{proof}


\subsection{Case $\alpha \ne 0$: an approximate Green function}


Let $\phi_{1,0}$ and $\phi_{2,0}$ be the two solutions of $Ray_0(\phi) = 0$ that are constructed above in Lemma \ref{lem-defphi012}. We note that  solutions of $Ray_0(\phi) = f$ tend to a constant
value as $z \to + \infty$ since $\phi_{1,0} \to U_+-c$. We now construct normal mode solutions to the Rayleigh equation with $\alpha \not=0$
\begin{equation}\label{Raya}
Ray_\alpha(\phi) = (U-c)(\dz^2 - \alpha^2)\phi - U'' \phi = f
\end{equation}
By looking at the spatially asymptotic limit of the Rayleigh equation, 
we observe that there are two normal mode solutions of \eqref{Raya}
 whose behaviors are as $e^{\pm\alpha z}$ at infinity. 
 In order to study the mode which behaves like $e^{-\alpha z}$ we introduce
\begin{equation}\label{def-phia12}
\phi_{1,\alpha } = \phi_{1,0} e^{-\alpha z} ,\qquad \phi_{2,\alpha} = \phi_{2,0} e^{-\alpha z}.
\end{equation}
A direct calculation shows that the Wronskian determinant 
$$
W[\phi_{1,\alpha},\phi_{2,\alpha}] =  \dz \phi_{2,\alpha} \phi_{1,\alpha} - \phi_{2,\alpha} \dz \phi_{1,\alpha}  = e^{-2\alpha z}
$$ is  non zero. In addition, we can check that 
\begin{equation}\label{Ray-phia12}
Ray_\alpha(\phi_{j,\alpha}) = - 2 \alpha (U-c) \dz \phi_{j,0} e^{-\alpha z} 
\end{equation}
We are then led to introduce an approximate Green function $G_{R,\alpha}(x,z)$ defined by 
$$
G_{R,\alpha}(x,z) = \left\{ \begin{array}{rrr} (U(x)-c)^{-1} e^{-\alpha (z-x)}  \phi_{1,0}(z) \phi_{2,0}(x), \quad \mbox{if}\quad z>x\\
(U(x)-c)^{-1} e^{-\alpha (z-x)}  \phi_{1,0}(x) \phi_{2,0}(z), \quad \mbox{if}\quad z< x.\end{array}\right.
$$
Again, like $G_{R,0}(x,z)$, the Green function $G_{R,\alpha}(x,z)$ is ``singular'' near $z = z_c$ with two sources 
of singularities: one arising from $1/ (U(x) - c)$ for $x$ near $ z_c$ and the other coming from the $(z - z_c) \log (z- z_c)$ singularity
of $\phi_{2,0}(z)$. By a view of \eqref{Ray-phia12}, it is clear that 
\begin{equation}\label{id-Gxz}
Ray_\alpha (G_{R,\alpha}(x,z)) = \delta_{x} -2\alpha (U- c) E_{R,\alpha}(x,z),
\end{equation}
for each fixed $x$. Here the error term $E_{R,\alpha}(x,z)$ is defined by  
$$
E_{R,\alpha}(x,z) = 
 \left\{ \begin{array}{rrr}
 (U(x)-c)^{-1} e^{-\alpha (z-x)} \partial_z \phi_{1,0}(z) \phi_{2,0}(x), \quad \mbox{if}\quad z>x\\
 (U(x)-c)^{-1}e^{-\alpha (z-x)}\  \phi_{1,0}(x) \partial_z \phi_{2,0}(z), \quad \mbox{if}\quad z< x.\end{array}\right.
$$
We then introduce an approximate inverse of the operator $Ray_\alpha$ defined by
\begin{equation}\label{def-RaySa}
RaySolver_\alpha(f)(z) 
:= \int_0^{+\infty} G_{R,\alpha}(x,z) f(x) dx
\end{equation}
and the error remainder 
\begin{equation}\label{def-ErrR}
Err_{R,\alpha}(f)(z) := 2\alpha (U(z) - c) \int_0^{+\infty} E_{R,\alpha}(x,z) f(x) dx
\end{equation}

%

\begin{lemma}\label{lem-RaySa} Assume that $\I c \not =0$, and let $p$ be $0,1,$ or $2$. For any $f\in {X_p^\eta}$,  with $\eta > \alpha$, the function $RaySolver_\alpha(f)$ is well-defined in $Y^\alpha_{p+2}$, satisfying 
$$ Ray_\alpha(RaySolver_\alpha(f)) = f + Err_{R,\alpha}(f).$$
Furthermore, there hold  
\begin{equation}\label{est-RaySa}
\| RaySolver_\alpha(f)\|_{Y^\alpha_{p+2}} \le C (1+|\log \I c|) \|f\|_{{X_p^\eta}},
\end{equation}
and 
\begin{equation}\label{est-ErrRa} 
\|Err_{R,\alpha}(f)\|_{Y_{p}^\eta} \le C\alpha  (1+|\log (\I c)|)  \|f\|_{X_p^\eta} ,
\end{equation}
for some universal constant $C$. 
\end{lemma}
\begin{proof}  The proof follows similarly to that of Lemmas \ref{lem-RayS0} and \ref{lem-derRayS0}. In fact, the proof of the right order of singularities near the critical layer follows identically from that of Lemmas \ref{lem-RayS0} and \ref{lem-derRayS0}.  

Let us check the right behavior at infinity. Consider the case $p=0$ and assume $\|f \|_{X_0^\eta} =1$. Similarly to the estimate \eqref{est-Gr0}, Lemma \ref{lem-defphi012} and the definition of $G_{R,\alpha}$ yield
  $$|G_{R,\alpha}(x,z)| \le  C e^{-\alpha (z-x)} \max\{ (1+x), |x-z_c|^{-1} \}.$$
Hence, by definition, 
$$ |RaySolver_\alpha (f)(z) |\le C e^{-\alpha z} \int_0^\infty e^{\alpha x} e^{-\eta x}\max\{ (1+x), |x-z_c|^{-1} \}\; dx $$ 
which is clearly bounded by $C(1+|\log \I c|) e^{-\alpha z}$.  This proves the right exponential decay of $RaySolver_\alpha (f)(z)$ at infinity, for all $f \in X_0^\eta$.

Next, by definition, we have 
$$\begin{aligned}
Err_{R,\alpha}(f)(z) &= -2\alpha (U(z) - c)  \partial_z \phi_{2,0}(z)   \int_z^\infty  e^{-\alpha (z-x)} \phi_{1,0}(x)\frac{f(x)}{U(x)-c}\; dx  
\\ & \quad - 2\alpha (U(z) - c)  \dz \phi_{1,0}(z) \int_0^ze^{-\alpha (z-x)} \phi_{2,0}(x) {f(x) \over U(x) - c} dx .
\end{aligned}$$
Since $f(z), \dz \phi_{1,0}(z)$ decay exponentially at infinity, the exponential decay of $Err_{R,\alpha}(f)(z)$ follows directly from the above integral representation. It remains to check the order of singularity near the critical layer. Clearly, for bounded $z$, we have 
$$ |E_{R,\alpha}(x,z) | \le C (1+ |\log (z-z_c)| ) e^{\alpha x} \max \{ 1, |x-z|^{-1}\} . $$
The lemma then follows at once, using the extra factor of $U-c$ in the front of the integral \eqref{def-ErrR} to bound the $\log (z-z_c)$ factor. The estimates for derivatives follow similarly. \end{proof}

%
%
%

\subsection{Case $\alpha \ne 0$: the exact solver for Rayleigh}


We now construct the exact solver for the Rayleigh operator by iteration. Let us denote 
$$ S_0(z) : = RaySolver_\alpha(f)(z),\qquad E_0(z): = Err_{R,\alpha}(f)(z).$$
It then follows that $Ray_\alpha (S_0) (z) = f(z) + E_0(z)$. Inductively, we define
$$ S_n(z): = - RaySolver_\alpha(E_{n-1})(z), \qquad E_n(z): = - Err_{R,\alpha}(E_{n-1})(z) ,$$
for $n\ge 1$. It is then clear that for all $n\ge 1$, 
\begin{equation}\label{eqs-appSn} Ray_\alpha  \Big( \sum_{k=0}^n S_k(z)\Big) = f(z) + E_n(z).\end{equation}
This leads us to introduce the exact solver for Rayleigh defined by 
\begin{equation}\label{def-exactRayS} RaySolver_{\alpha,\infty} (f) := RaySolver_\alpha(f)(z) - \sum_{n\ge 0} (-1)^n RaySolver_{\alpha} (E_n)(z).\end{equation}

\begin{proof}[Proof of Proposition \ref{prop-exactRayS}] By a view of \eqref{est-ErrRa}, we have 
$$\| E_n \|_\eta = \| (Err_{R,\alpha})^n (f)\|_\eta \le C^n\alpha^n  (1+|\log (\I c)|)^n  \|f\|_\eta,$$
which implies that $E_n\to 0$ in $X_\eta$ as $n \to \infty$ as long as $\alpha \log \I c$ is sufficiently small. In addition, by a view of  \eqref{est-RaySa},
$$\|RaySolver_{\alpha} (E_n)\|_{Y^\alpha_2} \le C C^n\alpha^n  (1+|\log (\I c)|)^n  \|f\|_\eta .$$
This shows that the series
$$\sum_{n\ge 0} (-1)^n RaySolver_{\alpha} (E_n)(z)$$
converges in $Y_2^\alpha$, assuming that $\alpha \log \I c$ is small.

Next, by taking the limit of $n\to \infty$ in \eqref{eqs-appSn}, the equation \eqref{eqs-RaySolver} holds by definition at least in the distributional sense. The estimates when $z$ is near $z_c$ follow directly from the similar estimates on $RaySolver_\alpha(\cdot)$; see Lemma \ref{lem-RaySa}. The proof of Proposition \ref{prop-exactRayS} is thus complete. 
\end{proof}

\subsection{Exact Rayleigh solutions}\label{sec-exactRayleigh}
We shall construct two independent exact Rayleigh solutions by iteration, starting from the approximate Rayleigh solutions $\phi_{j,\alpha}$ defined as in \eqref{def-phia12}. 

\begin{lemma}\label{lem-exactphija} For $\alpha$ small enough so that $\alpha |\log \I c| \ll 1$, 
there exist two independent functions $\phi_{Ray,\pm} \in e^{\pm \alpha z}L^\infty$ such that
$$
Ray_\alpha ( \phi_{Ray,\pm} ) = 0, \qquad W[\phi_{Ray,+},\phi_{Ray,-}](z) = \alpha.
$$
Furthermore, we have the following expansions in $L^\infty$: 
$$
\begin{aligned}
\phi_{Ray,-} (z)&=  e^{-\alpha z} \Big (U-c + O(\alpha )\Big).
\\
\phi_{Ray,+} (z)&=  e^{\alpha z} \cO(1),
\end{aligned}$$
as $z\to \infty$. At $z = 0$, there hold
$$ 
\begin{aligned}
\phi_{Ray,-}(0) &= U_0 - c + \alpha (U_+-U_0) ^2  \phi_{2,0}(0) + \cO(\alpha(\alpha + |z_c|))
\\
\phi_{Ray,+}(0) &= \alpha  \phi_{2,0}(0) +  \cO(\alpha^2)
\end{aligned}$$
with $\phi_{2,0}(0) =  {1 \over {U'_c}}  -  {2U''_c\over {U'_c}^2 } z_c \log z_c + \cO(z_c)$. 
 \end{lemma}
\begin{proof} 
Let us start with the decaying solution $\phi_{Ray,-}$, which is now constructed by induction. Let us introduce 
$$ \psi_{0} =  e^{-\alpha z} (U-c), \qquad e_{0} = - 2\alpha (U-c) U' e^{-\alpha z},$$
and inductively for $k \ge 1$, 
$$ \psi_{k} = - RaySolver_\alpha (e_{k-1}), \qquad e_{k} = - Err_{R,\alpha} (e_{k-1}).$$
We also introduce 
$$ \phi_{N} = \sum_{k=0}^N \psi_{k} .$$
By definition, it follows that 
$$ Ray_\alpha (\phi_{N}) = e_{N}, \qquad \forall ~N\ge 1.$$

%

We observe that $\|e_{0}\|_{\eta+\alpha} \le C \alpha$ and $\|\psi_{0}\|_\alpha \le C$. Inductively for $k\ge 1$, by the estimate \eqref{est-ErrRa}, we have 
$$\| e_{k}\|_{\eta + \alpha}  \le C\alpha  (1+|\log (\I c)|)  \|e_{k-1}\|_{\eta+\alpha} \le C \alpha(C\alpha  (1+|\log (\I c)|))^{k-1} ,
$$
and by Lemma \ref{lem-RaySa}, 
$$\| \psi_{k} \|_\alpha \le C(1+|\log (\I c)|)  \| e_{k-1}\|_{\eta + \alpha} \le (C\alpha  (1+|\log (\I c)|))^{k} .$$
Thus, for sufficiently small $\alpha$, the series $\phi_{N}$ converges in $X_\alpha$ and the error term $e_{N}\to 0$ in $X_{\eta+\alpha}$. This proves the existence of the exact decaying Rayleigh solution $\phi_{Ray,-}$ in $X_\alpha$, or in $e^{-\alpha z}L^\infty$.

As for the growing solution, we simply define 
$$ \phi_{Ray,+}  =\alpha \phi_{Ray,-}(z) \int_{1/2} ^ z \frac{1}{\phi^2_{Ray,-} (y) }\; dy.$$
By definition, $\phi_{Ray,+} $ solves the Rayleigh equation identically. Next, since $\phi_{Ray,-}(z)$ tends to $e^{-\alpha z} (U_+ - c + \cO(\alpha))$, $\phi_{Ray,+} $ is of order $e^{\alpha z}$ as $z \to \infty$.

Finally, at $z=0$, we have 
$$
\begin{aligned}
\psi_1(0) &= - RaySolver_\alpha(e_0) (0) = - \phi_{2,\alpha}(0) \int_0^{+\infty} e^{2\alpha x}\phi_{1,\alpha}(x) {e_0(x) \over U(x) - c} dx 
\\
&=  2 \alpha \phi_{2,0}(0) \int_0^{+\infty}  U' (U-c)dz = \alpha (U_+-U_0) (U_+ + U_0 - 2c) \phi_{2,0}(0)
\\
&  = \alpha (U_+-U_0)^2  (U_+ + U_0 - 2c) \phi_{2,0}(0) +  2\alpha (U_+-U_0) (U_0 - c) \phi_{2,0}(0).
\end{aligned}
$$
From the definition, we have $\phi_{Ray,-}(0) = U_0 -c + \psi_1(0) + \cO(\alpha^2)$. 
This proves the lemma, upon using that $U_0 - c = \cO(z_c)$. 
\end{proof}




\newpage
\section{Airy equations}\label{sec-mAiry}

Our ultimate goal is to inverse the Airy operator defined as in \eqref{opAiry}, and thus we wish to construct the Green function for the primitive Airy equation
\beq \label{Airyp}
Airy(\phi): = \eps \partial_z^4 \phi - (U(z)-c + 2 \eps \alpha^2) \dz^2 \phi = 0.
\eeq

\subsection{Classical Airy equations}\label{sec-Airy}


The classical Airy functions play a major role in the analysis near the critical layer. The aim of this section is to recall some properties of the classical Airy functions. The classical Airy equation is
\beq \label{Airy}
\dz^2 \phi - z \phi = 0, \qquad z \in \CC.
\eeq
In connection with the Orr-Somerfeld equation with $\epsilon$ being complex, we are interested in the Airy functions with argument 
$$
z = e^{i \pi / 6} x, \quad x \in \rit .
$$
Let us state precisely what we will be needed. These classical results can be found in \cite{Airy,Vallee}; see also \cite[Appendix]{Reid}.

\begin{lemma}\label{lem-classicalAiry} The classical Airy equation \eqref{Airy}
has two independent solutions $Ai(z)$ and 
$Ci(z)$ so that the Wronskian determinant of $Ai$ and $Ci$ equals
\beq \label{WCiAi}
W(Ai, Ci)  =  Ai(z) Ci'(z) -Ai'(z) Ci(z) =   1 .
\eeq
In addition, $Ai(e^{i \pi /6} x)$ and $Ci(e^{i \pi /6} x)$ converge to $0$ as $x\to \pm \infty$
($x$ being real), respectively. Furthermore, there hold asymptotic bounds: 
\beq \label{boundAik1}
\Bigl| Ai(k, e^{i \pi / 6} x) \Bigr| \le
 {C| x |^{-k/2-1/4} } e^{-\sqrt{2 | x|} x / 3}, \qquad k\in \ZZ, \quad x\in \RR,
 \eeq
 and
 \beq \label{boundCik1}
\Bigl| Ci(k, e^{i \pi / 6} x) \Bigr| \le
 {C| x |^{-k/2-1/4} } e^{\sqrt{2 | x|} x / 3}, \qquad k\in \ZZ,\quad x\in \RR,
 \eeq
in which $Ai(0,z) = Ai(z)$, $Ai(k,z) = \partial_z^{-k} Ai(z)$ for $k\le 0$, and $Ai(k,z)$ is the $k^{th}$ primitives of $Ai(z)$ for $k\ge 0$ and is defined by the inductive path integrals
$$
Ai(k, z ) = \int_\infty^z Ai(k-1, w) \; dw 
$$
so that the integration path is contained in the sector with $|\arg(z)| < \pi/3$. The Airy functions $Ci(k,z)$ for $k\not =0$ are defined similarly.  

\end{lemma}

The following lemma whose proof can again 
be found in the mentioned physical references will be of use in the latter analysis. 
\begin{lemma}\label{lem-expAi12} Let $S_1$ be the sector in the complex plane such that the argument is between $2\pi/3$ and $4\pi /3$. There hold expansions
$$ 
\begin{aligned}
Ai(1,z) \quad &\approx \quad  - \frac{1}{2 \sqrt \pi z^{3/4}} e^{- \frac 23 z^{3/2}} (1 + \cO(|z|^{-3/2}))
\\
Ai(2,z) \quad &\approx \quad  \frac{1}{2 \sqrt \pi z^{5/4}} e^{- \frac 23 z^{3/2}} (1 + \cO(|z|^{-3/2}))
\end{aligned}
$$
for all large $z$ in $S_1$. In addition, at $z = 0$, there holds
$$ Ai(k,0) = \frac{(-1)^k 3^{-(k+2)/3}}{\Gamma (\frac{k+2}{3})} , \qquad k \in \ZZ, $$
in which $\Gamma(\cdot)$ is the Gamma function defined by $\Gamma(z) = \int_0^\infty t^{z-1} e^{-t}dt$. 
\end{lemma}

 \subsection{Langer transformation} \label{sec-Langer}


Since the profile $U$ depends on $z$ in a non trivial manner, we make a change of variables and unknowns in order to go back
to classical Airy equations studied in the previous section. This change is very classical in physical literature, and
called the Langer's transformation.

\begin{definition}\label{def-Langer} By Langer's transformation $(z,\phi) \mapsto (\eta,\Phi)$, we mean $\eta = \eta(z)$ defined by
\begin{equation}\label{var-Langer}
\eta (z) = \Big[ \frac 32 \int_{z_c}^z \Big( \frac{U-c}{U'_c}\Big)^{1/2} \; dz \Big]^{2/3}
\end{equation}
and $\Phi = \Phi(\eta)$ defined by the relation
\begin{equation}\label{phi-Langer}
 \dz^2 \phi (z) = \dot z ^{1/2} \Phi(\eta),
 \end{equation}
 in which $\dot z = \frac{d z( \eta)}{d \eta} $ and $z = z(\eta)$ is the inverse of the map $\eta = \eta(z)$. 
\end{definition}
Direct calculation gives a useful fact $(U-c)\dot z^2  = U'_c \eta$. Next, using that $c = U(z_c)$, one observes that for $z$ near $z_c$, we have 
\begin{equation}\label{est-eta}\begin{aligned}
\eta (z) &=  \Big[ \frac 32 \int_{z_c}^z \Big(z-z_c + \frac{U_c''}{U_c'} (z-z_c)^2 
+ \cO(|z-z_c|^3)\Big)^{1/2} \; dz \Big]^{2/3} \\
&= z-z_c + \frac {1}{10} \frac{U_c''}{U_c'} (z-z_c)^2 + \cO(|z-z_c|^3).\end{aligned}
\end{equation}
In particular, we have \begin{equation}\label{est-edot}\eta'(z) = 1 + \cO(|z-z_c|),\end{equation}
and thus the inverse $ z = z(\eta)$ is locally well-defined and locally increasing near $z=z_c$. In addition, 
$$
\dot z = \frac{1}{\eta'(z)} = 1 + \cO(|z-z_c|).
$$
Next, we note that $\eta'(z)^2 = \frac{ U-c}{ U_c' \eta(z)}$, which is nonzero away from $z = z_c$. Thus, the inverse of $\eta = \eta(z)$ exists for all $z\ge 0$.  

In addition, by a view of the definition \eqref{var-Langer} and the fact that $(U-c)\dot z^2  = U'_c \eta$, we have 
\begin{equation}\label{asymp-eta} |\eta(z)|\le C (1+|z|)^{2/3}, \qquad |\eta'(z)| \le C (1+|z|)^{-1/3} ,\qquad  |\dot z(\eta(z))| \le C (1+ |z|)^{1/3}, \end{equation}
for some universal constant $C$.

The following lemma links \eqref{Airyp} with the classical Airy equation. 

\begin{lemma}\label{lem-Langer} Let $(z,\phi) \mapsto (\eta, \Phi)$ be the Langer's transformation defined as in Definition \ref{def-Langer}. Assume that $\Phi(\eta)$ solves 
$$\epsilon   \partial^2_\eta \Phi  - U_c' \eta \Phi   = f(\eta).$$
Then, $\phi = \phi(z)$ solves
$$Airy ( \phi) =  \dot z ^{-3/2} f(\eta(z))+  \epsilon [  \D_z^2 \dot z^{1/2} \dot z^{-1/2} - 2\alpha^2 ]\dz^2 \phi (z) $$
\end{lemma}
\begin{proof} Derivatives of the identity $ \dz^2 \phi (z) = \dot z ^{1/2} \Phi(\eta)$
are $$\begin{aligned}
\dz^3 \phi(z) &= \dot z ^{-1/2} \partial_\eta \Phi  + \D_z \dot z^{1/2} \Phi
\end{aligned}$$
and 
\begin{equation}\label{dz4-phiLg}
\begin{aligned}\dz^4 \phi(z) &= \dot z ^{-3/2} \partial^2_\eta \Phi  + [\D_z \dot z^{-1/2} + \dot z^{-1} \D_z \dot z^{1/2}] \D_\eta \Phi + \D_z^2 \dot z^{1/2} \Phi
\\&= \dot z ^{-3/2} \partial^2_\eta \Phi  + \D_z^2 \dot z^{1/2} \Phi
.\end{aligned}
\end{equation}
This proves that 
\begin{equation}\label{dz2-cal}
\dz^2 (\dot z^{1/2} \Phi(\eta)) = \dot z^{-3/2} \partial_\eta^2 \Phi(\eta) + \dz^2 \dot z^{1/2} \Phi(\eta) 
\end{equation}
Putting these together and using the fact that $(U-c)\dot z^2  = U'_c \eta$, we get 
$$\begin{aligned}
\eps \partial_z^4 \phi - (U(z)-c) \dz^2 \phi  &= \epsilon \dot z ^{-3/2} \partial^2_\eta \Phi  
- (U-c) \dot z^{1/2}  \Phi  + \epsilon  \D_z^2 \dot z^{1/2} \Phi
\\
&= \dot z ^{-3/2} f(\eta)+ \epsilon  \D_z^2 \dot z^{1/2} \dot z^{-1/2}\dz^2 \phi (z).
\end{aligned}$$
The lemma follows. 
\end{proof}

%
%


\subsection{Resolution of the modified Airy equation}

In this section we will construct the Green function for the Airy equation:
\beq \label{Airy-cl}
\mathcal{A}_\mathrm{a}(\Phi): = \eps \partial_z^2 \Phi - (U(z)-c) \Phi = f.
\eeq
Let us denote 
$$
\delta = \Bigl( { \eps \over U_c'} \Bigr)^{1/3}  = e^{-i \pi / 6} (\alpha R U_c')^{-1/3},
$$
and introduce the notation $X = \delta^{-1} \eta(x)$ and $Z = \delta^{-1} \eta(z) $, where $\eta(z)$ is the Langer's variable defined as in \eqref{var-Langer}. We define an approximate Green function for the Airy equation:
\begin{equation}\label{def-Gcl} 
G_\mathrm{a}(x,z) =  i \delta \pi \eps^{-1}  \dot x\left\{ \begin{array}{rrr}  Ai(X)Ci(Z), \qquad &\mbox{if}\qquad x>z,\\
 Ai (Z) Ci(X) , \qquad &\mbox{if}\qquad x<z,
\end{array}\right.
\end{equation}with $\dot x = \dot z (\eta (x))$. 
It follows that $G_\mathrm{a}(x,z)$ satisfies the jump conditions across $x=z$:
$$[G_\mathrm{a}(x,z)]_{\vert_{x=z}} = 0 , \qquad [\epsilon \partial_zG_\mathrm{a}(x,z)]_{\vert_{x=z}} = 1.$$
By definition, we have 
\begin{equation}\label{eqs-Gcl}
\eps \partial_z^2 G_\mathrm{a}(x,z)  - (U-c) G_\mathrm{a}(x,z)  = \delta_x(z) + E_\mathrm{a}(x,z),  
\end{equation} with $E_\mathrm{a}(x,z) = i\pi \eta''(z) \dot x Ai(X) Ci'(Z)$. 

Let us detail some estimates on $G_\mathrm{a}$ as a warm up for the following sections. Let us consider the case: $x<z$. By the estimates on the Airy functions obtained from Lemma \ref{lem-classicalAiry}, we get 
\begin{equation}\label{asymp-bd}
| \D_z^k Ai(e^{i \pi / 6} z)  \D_x^\ell Ci(e^{i \pi / 6} x) | \le 
{C|z|^{k/2-1/4} |x|^{\ell/2-1/4}} 
\exp \Bigl( {1\over 3} \sqrt{2 |x|} x -  {1\over 3} \sqrt{2 |z|} z \Bigr) ,
\end{equation} for $k \ge 0$, and for $x,z$ bounded away from zero.  Similar bounds can easily obtained for the case $x>z$. We remark that the polynomial growth in $x$ in the above estimate can be easily replaced by the growth in $z$, up to an exponentially decaying term.  

%

We obtain the following lemma. 

\begin{lemma} \label{lem-ptGreenbound-cl} Let $G_\mathrm{a}(x,z)$ be the approximate Green function defined as in \eqref{def-Gcl}, and $E_a(x,z)$ as defined in \eqref{eqs-Gcl}. Also let $X = \eta(x)/\delta$ and $Z = \eta(z)/\delta$. For $k,\ell = 0,1$, there hold pointwise estimates
\begin{equation}\label{Gcl-xnearz}
\begin{aligned}
|\D_z^\ell \D_x^k G_\mathrm{a}(x,z) | 
&\le 
C \delta^{-2-k-\ell}(1+|z|)^{(1-k-\ell)/3}(1+|Z|)^{(k+\ell-1)/2}
e^  {-  {\sqrt{2} \over 3} \sqrt{|Z|}|X-Z| }  ,
\\
|\D_z^\ell \D_x^k E_\mathrm{a}(x,z) | 
&\le 
C \delta^{-k-\ell}(1+|z|)^{(-3-k-\ell)/3}(1+|Z|)^{(k+\ell)/2}
e^  {-  {\sqrt{2} \over 3} \sqrt{|Z|}|X-Z| }  .
\end{aligned}
\end{equation} 
\end{lemma}
\begin{proof} The lemma follows directly from \eqref{asymp-bd}, upon noting that he pre-factor in terms of the lower case $z$ is due to the Langer's change of variables.  
\end{proof}

Let us next give a few convolution estimates. 
\begin{lemma}  \label{lem-ConvAiry-cl}  Let $G_\mathrm{a}(x,z)$ be the approximate Green function defined as in \eqref{def-Gcl}, and $E_a(x,z)$ as defined in \eqref{eqs-Gcl}. Also let $f\in X_\eta$, for some $\eta>0$.  
Then there is some constant $C$ so that 
\begin{equation}\label{conv-loc}\begin{aligned}
\Big\|\int_{0}^{\infty}  \partial_z^kG_\mathrm{a}(x,\cdot)  f(x)   dx \Big\|_\eta &\le  C \delta^{-1-k}  \|f\|_\eta 
\end{aligned}\end{equation}
and \begin{equation}\label{conv-nonloc}\begin{aligned}
\Big \|\int_{0}^{\infty} \partial_z^kE_\mathrm{a}(x,\cdot)  f(x)   dx\Big \|_\eta &\le C \delta^{1-k}  \|f\|_\eta 
\end{aligned}\end{equation}  for $k = 0,1$. 
\end{lemma} 


\begin{proof} Without loss of generality, we assume $\|f \|_\eta =1$. For $k=0,1$, using the bounds from Lemma \ref{lem-ptGreenbound-cl} and noting that $Z = \eta(z)/\delta \approx (1+|z|)^{2/3}/\delta$ as $z$ becomes large, we obtain 
$$\begin{aligned}
\int_{0}^{\infty} &|  \partial_z^kG_\mathrm{a}(x,z) f(x)|   dx
\\ &\le C \delta^{-2-k} \int_0^\infty (1+ z)^{(1-k)/3} e^{-\eta z}e^{-  {\sqrt{2} \over 3} \sqrt{|Z|}|X-Z| }  \; dx 
\\&\le C \delta^{-2-k} (1+ z)^{1-k/3} e^{-\eta z}   \int_0^\infty (1+x)^{-1/3} e^{-  {\sqrt{2} \over 3} \sqrt{|Z|}|X-Z| }   \delta dX 
 \\&\le C \delta^{-1-k}  e^{-\eta |z|}.
\end{aligned}$$
Here, we have used the change of variable $dx = \delta \dot z^{-1} dX$ with $\dot z \approx (1+|x|)^{1/3}$. Similar estimates hold for $E_\mathrm{a}(x,z)$. This completes the proof of the lemma.
\end{proof}

An approximate solution $\Phi$ of (\ref{Airy-cl}) is given by the convolution
\beq \label{GreenAiry-cl}
 \mathcal{A}^{-1}_\mathrm{a}(f) =  \int_0^{+\infty} G_\mathrm{a}(x,z) f(x) dx .
\eeq
Indeed, a direct calculation yields 
$$ 
\mathcal{A}_\mathrm{a} (\mathcal{A}^{-1}_\mathrm{a} (f)) = f + Err_\mathrm{a} (f), 
$$
with the error term defined by 
$$ Err_\mathrm{a}(f) = \int_0^\infty E_\mathrm{a}(x,z) f(x) \; dx .$$
The convolution lemma (Lemma \ref{lem-ConvAiry-cl}) in particular yields 
\begin{equation}\label{Err-cl} \| Err_\mathrm{a}(f)  \|_\eta  \le C\delta \|f \|_\eta,\end{equation}
for all $f\in X_\eta$. That is, $Err_\mathrm{a}(f)$ is indeed of order $\cO(\delta)$ in $X_\eta$. For this reason, we may now define by iteration an exact solver for the Airy operator $\mathcal{A}_\mathrm{a}(\cdot)$. Let us start with a fixed $f \in X_\eta$. Let us define
\begin{equation}\label{iter-Aphi}
\begin{aligned}
\phi_n &= - \mathcal{A}^{-1}_\mathrm{a}(E_{n-1})
\\
E_n &=  - Err_\mathrm{a}(E_{n-1}) \end{aligned}
\end{equation}
for all $n \ge 1$, with $E_0 = f$. Let us also denote 
$$
S_n = \sum_{k=1}^n \phi_k .
$$
It  follows by induction that 
$$ \mathcal{A}_\mathrm{a}(S_n) = f + E_n,$$
for all $n\ge 1$. Now by \eqref{Err-cl}, we have 
$$ \| E_n\|_\eta \le C \delta \|E_{n-1}\|_\eta \le (C\delta)^n \| f\|_\eta.$$ 
This proves that $E_n \to 0$ in $X_\eta$ as $n \to \infty$ since $\delta$ is small. In addition, by a view of Lemma \ref{lem-ConvAiry-cl}, we have 
$$ \| \phi_n\|_{\eta} \le C \delta^{-1} \| E_{n-1}\|_\eta \le C \delta^{-1}(C\delta)^{n-1} \|f\|_\eta.$$
This shows that $\phi_n$ converges to zero in $X_{\eta}$ as $n \to \infty$, and furthermore the series 
$$ S_n \to S_\infty$$
 in $X_{\eta}$ as $n \to \infty$, for some $S_\infty \in X_{\eta}$. We then denote $\mathcal{A}^{-1}_{a,\infty}(f) = S_\infty$, for each $f \in X_\eta$. In addition, we have $ \mathcal{A}_\mathrm{a} (S_\infty) = f,$ that is, $\mathcal{A}^{-1}_{a,\infty}(f) $ is the exact solver for the modified Airy operator. A similar estimate follows for derivatives. 

To summarize, we have proved the following proposition. 
\begin{prop}\label{prop-exactAiry-cl} Assume that $\delta$ is sufficiently small. There exists an exact solver $\mathcal{A}^{-1}_{a,\infty}(\cdot)$ as a well-defined operator from $X_\eta$ to $X_\eta$, for arbitrary fixed $\eta>0$, so that 
$$ \mathcal{A}_\mathrm{a}(\mathcal{A}^{-1}_{a,\infty} (f)) = f.$$
In addition, there holds 
$$
\| \mathcal{A}^{-1}_{a,\infty} (f) \|_{X^\eta_k} \le C \delta^{-1-k} \| f \|_\eta ,\qquad k = 0,1,
$$
for some positive constant $C$. 
\end{prop}

\subsection{An approximate Green function of primitive Airy equation}


In this section we will construct an approximate Green function for (\ref{Airyp}). By a view of the Langer's transformation, let us introduce an auxiliary Green function
$$ 
G_\mathrm{aux}(X,Z) = i \delta \pi \eps^{-1}   \left\{ \begin{array}{rrr} 
Ai(X)Ci(Z), \qquad &\mbox{if}\qquad \xi > \eta,\\
 Ai (Z) Ci(X) , \qquad &\mbox{if}\qquad \xi < \eta.
\end{array}\right.
$$
By definition, we have 
\begin{equation}\label{eqs-mGa}
\eps \partial_\eta^2 G_\mathrm{aux}(X,Z)  - U_c' \eta G_\mathrm{aux}(X,Z)  = \delta_\xi(\eta).  
\end{equation}
Next, let us take $\xi = \eta(x)$ and $\eta = \eta(z)$, where $\eta(\cdot)$ is the Langer's transformation and denote $\dot x = 1/\eta '(x)$ and $\dot z = 1/ \eta'(z)$. By a view of \eqref{phi-Langer}, we define the function $G(x,z)$ so that 
\beq \label{def-doubleG}
\dz^2G(x,z) =\dot x ^{3/2} \dot z^{1/2}  G_\mathrm{aux}(\delta^{-1}\eta(x),\delta^{-1}\eta(z)),
\eeq 
in which the factor $\dot x ^{3/2}$ was added simply to normalize the jump of $G(x,z)$. It then follows from Lemma \ref{lem-Langer} together with $\delta_{\eta(x)} (\eta(z)) = \delta_x(z)$ that  
\begin{equation}\label{eqs-2G}
Airy (G(x,z)) =  \delta_x(z) + \epsilon  (\D_z^2 \dot z^{1/2} \dot z^{-1/2} - 2\alpha^2)\dz^2G(x,z). 
\end{equation}

That is, $G(x,z)$ is indeed an approximate Green function of the primitive Airy operator $\epsilon \dz^4 - (U-c)\dz^2$ up to a small error term of order $\epsilon \dz^2 G = \cO(\delta)$. It remains to solve \eqref{def-doubleG} for $G(x,z)$, retaining the jump conditions on $G(x,z)$ across $x=z$. 

In view of primitive Airy functions, let us denote
$$ \widetilde Ci(1,z) =\delta^{-1} \int_0^z \dot y^{1/2} Ci(\delta^{-1}\eta(y))\; dy, \qquad \widetilde Ci(2,z) = \delta^{-1}\int_0^z \widetilde Ci(1,y)\; dy$$
and 
$$ \widetilde Ai(1,z) = \delta^{-1}\int_\infty^z \dot y^{1/2} Ai(\delta^{-1}\eta(y))\; dy, \qquad \widetilde Ai(2,z) =\delta^{-1} \int_\infty^z \widetilde Ai(1,y)\; dy.$$
Thus, together with our convention that the Green function $G(x,z)$ should vanish as $z$ goes to $+\infty$ for each fixed $x$, we are led to introduce \beq \label{def-GreenAiry2} 
G(x,z) = i \delta^3 \pi \epsilon^{-1}  \dot x^{3/2} \left\{ \begin{aligned}  
 \Big[ Ai(\delta^{-1}\eta(x)) \widetilde Ci(2,z) + \delta^{-1} a_1 (x) (z-x) +  a_2(x)  \Big] , &\mbox{if }x>z,\\
  Ci(\delta^{-1}\eta(x)) \widetilde Ai(2,z) ,  &\mbox{if } x<z,
\end{aligned} \right.
\eeq
in which $a_1(x), a_2(x)$ are chosen so that the jump conditions (see below) hold. Clearly, by definition, $G(x,z)$ solves \eqref{def-doubleG}, and hence \eqref{eqs-2G}. Here the jump conditions on the Green function read:
\begin{equation}\label{def-jumpG0}
\begin{aligned}
~[G(x,z)]_{\vert_{x=z}} =  [\partial_z G(x,z)]_{\vert_{x=z}}  = [\partial_z^2 G(x,z)]_{\vert_{x=z}}  =0
\end{aligned}\end{equation}
and 
\begin{equation}\label{def-jumpG1}
\begin{aligned}
~[\epsilon \partial_z^3G(x,z)]_{\vert_{x=z}} =  1.
\end{aligned}\end{equation}
We note that from \eqref{def-doubleG} and the jump conditions on $G_\mathrm{aux}(X,Z)$ across $X=Z$, the above jump conditions of $\partial_z^2 G$ and $\partial_z^3 G$ follow easily.  In order for the jump conditions on $G(x,z)$ and $\partial_zG(x,z)$, we take 
\begin{equation}\label{def-ta12}
\begin{aligned}   a_1 (x)  &=
 Ci(\delta^{-1}\eta(x)) \widetilde Ai(1,x) - Ai(\delta^{-1}\eta(x)) \widetilde Ci(1,x)  ,\\
a_2(x)  
&=  Ci(\delta^{-1}\eta(x))\widetilde Ai(2,x) - Ai(\delta^{-1}\eta(x)) \widetilde Ci(2,x)  .
\end{aligned}
\end{equation}

We obtain the following lemma. 

\begin{lemma}\label{lem-GreenAiry} Let $G(x,z)$ be defined as in \eqref{def-GreenAiry2}. Then $G(x,z)$ is an approximate Green function of the Airy operator (\ref{Airyp}). Precisely, there holds
\beq \label{approxGreen}
Airy (G(x,z) ) = \delta_x(z) + Err_A(x,z)
\eeq
for each fixed $x$, where $Err_A(x,z)$ denotes the error kernel defined by 
\begin{equation}\label{def-ErrA} \begin{aligned}
Err_A(x,z)  &=   \epsilon  (\D_z^2 \dot z^{1/2} \dot z^{-1/2} - 2\alpha^2)\dz^2G(x,z). 
\end{aligned}\end{equation}
\end{lemma}

It appears convenient to denote by $\widetilde G(x,z)$ and $E(x,z)$ the localized and non-localized part of the Green function, respectively. Precisely, we denote 
$$
\widetilde G(x,z) = i \delta^3 \pi \epsilon^{-1}  \dot x^{3/2} \left\{ \begin{array}{rrr}  Ai(\delta^{-1}\eta(x)) \widetilde Ci(2,z) , &\mbox{if }x>z,\\
  Ci(\delta^{-1}\eta(x)) \widetilde Ai(2,z) ,  &\mbox{if } x<z,
\end{array} \right. 
$$ and 
$$
E(x,z) = i \delta^3 \pi \epsilon^{-1}  \dot x^{3/2} \left\{ \begin{array}{rrr}  
\delta^{-1} a_1 (x) (z-x) +  a_2(x) , &\mbox{if }x>z,\\
 0,  &\mbox{if } x<z.
\end{array} \right.
$$

Let us give some bounds on the Green function, using the known bounds on $Ai(\cdot)$ and $Ci(\cdot)$. We have the following lemma. 

\begin{lemma} \label{lem-ptGreenbound} Let $G(x,z) = \widetilde G(x,z) + E(x,z)$ be the Green function defined as in \eqref{def-GreenAiry2}, and let $X = \eta(x)/\delta$ and $Z = \eta(z)/\delta$. There hold pointwise estimates
\begin{equation}\label{mG-xnearz}
\begin{aligned}
|\D_z^\ell \D_x^k \widetilde G(x,z) | 
\le 
C \delta^{-k-\ell}(1+|z|)^{(4-k-\ell)/3}(1+|Z|)^{(k+\ell-3)/2}
e^  {-  {\sqrt{2} \over 3} \sqrt{|Z|}|X-Z| }  .
\end{aligned}
\end{equation} 
Similarly, for the non-localized term, we have 
\begin{equation}\label{non-locE}
|E(x,z) |\le   C(1+x)^{4/3} (1+|X|)^{-3/2} + C(1+x)^{1/3}|x-z|,
\end{equation}
for $x>z$. \end{lemma}
\begin{proof} We recall that for $z$ near $z_c$, we can write $ \dot z (\eta(z) ) = 1 + \cO(|z-z_c|),$ which in particular yields that $\frac 12 \le \dot z(\eta(z)) \le \frac 32 $ for $z$ sufficiently near $z_c$. In addition, by a view of the definition \eqref{var-Langer}, $\eta(z)$ grows like  $ (1+|z|)^{2/3}$ as $z\to \infty$; see \eqref{asymp-eta}.  

Let us assume that $z\ge 1$. It suffices to give estimates on $\widetilde Ai(k,z), \widetilde Ci(k,z)$. With notation $Y = \eta(y)/|\delta|$, we have 
$$\begin{aligned}
 |\widetilde Ai(1,z) | &\le  \delta^{-1}\int_z^\infty |\dot y^{1/2}Ai(e^{i\pi/6}Y) |\; dy \le  C\delta^{-1}\int_z^\infty (1+y)^{1/6} (1+|Y|)^{-1/4} e^{-\sqrt{2|Y|} Y/3} \; dy 
 \\& \le  C\int_Z^\infty (1+y)^{1/6} (1+|Y|)^{-1/4} e^{-\sqrt{2|Y|} Y/3} \; (1+y)^{1/3}dY 
 \\& \le  C\int_Z^\infty (1+y)^{1/2} (1+|Y|)^{-1/4} e^{-\sqrt{2|Z|} Y/3} \; dY 
 \\& \le  C(1+z)^{1/2} (1+|Z|)^{-3/4} e^{-\sqrt{2|Z|} Z/3}  ,
  \end{aligned}
$$
and 
$$\begin{aligned}
 |\widetilde Ai(2,z) | &\le  \delta^{-1}\int_z^\infty |\widetilde Ai(1,y) |\; dy \le  C\delta^{-1}\int_z^\infty (1+y)^{1/2} (1+|Y|)^{-3/4} e^{-\sqrt{2|Y|} Y/3} \; dy 
 \\& \le  C\int_Z^\infty (1+y)^{5/6} (1+|Y|)^{-3/4} e^{-\sqrt{2|Z|} Y/3} \; dY 
 \\& \le  C(1+z)^{5/6} (1+|Z|)^{-5/4} e^{-\sqrt{2|Z|} Z/3}  .
  \end{aligned}
$$
Similarly, we have 
$$ \begin{aligned}
|\widetilde Ci(1,z)| &\le \delta^{-1} \int_0^z |\dot y^{1/2}Ci(e^{i\pi/6}Y)|\; dy
\le C\delta^{-1}\int_0^z (1+y)^{1/6} (1+|Y|)^{-1/4} e^{\sqrt{2|Y|} Y/3} \; dy
\\& \le C\int_0^z (1+z)^{1/2} (1+|Y|)^{-1/4} e^{\sqrt{2|Y|} Y/3} \; dY
\\& \le C (1+z)^{1/2} (1+|Z|)^{-3/4} e^{\sqrt{2|Z|} Z/3} 
\end{aligned}$$
and 
$$ \begin{aligned}
|\widetilde Ci(2,z)| &\le \delta^{-1} \int_0^z |\widetilde Ci(1,y)|\; dy
\le C\int_0^z (1+y)^{5/6} (1+|Y|)^{-3/4} e^{\sqrt{2|Y|} Y/3} \; dY
\\& \le C (1+z)^{5/6} (1+|Z|)^{-5/4} e^{\sqrt{2|Z|} Z/3}. 
\end{aligned}$$

In the case that $z\le 1$, the above estimates remain valid. Indeed, first consider the case that $z\ge \R z_c$. In this case, we still have $|Y|\ge |Z|$ whenever $y\ge z$, and so the estimates on $\widetilde Ai(k,z)$ follow in the same way as done above. Next, consider the case that $z\le \R z_c$. In this case, we have 
$$\begin{aligned}
 |\widetilde Ai(1,z) | \le (1+|Z|)^{-3/4} e^{\sqrt{2} |Z|^{3/2}/3} , \qquad  |\widetilde Ai(2,z) | \le  C (1+|Z|)^{-5/4} e^{\sqrt{2}|Z|^{3/2} /3}  .
  \end{aligned}
$$
That is, like $Ai(\eta(z)/\delta)$, the functions $\widetilde Ai(k,z)$ grow exponentially fast as $z$ tends to zero and is away from the critical layer.  Similarly, we also have 
$$ \begin{aligned}
|\widetilde Ci(1,z)| \le C(1+|Z|)^{-3/4} e^{ - \sqrt{2}|Z|^{3/2}/3}, \qquad 
|\widetilde Ci(2,z)| \le C (1+|Z|)^{-5/4}e^{ - \sqrt{2}|Z|^{3/2}/3},
\end{aligned}$$
for $z\le \R z_c$. The estimates become significant when the critical layer is away from the boundary layer, that is when $\delta \ll |z_c|$. 

By combining together these bounds and those on $Ai(\cdot)$, $Ci(\cdot)$, the claimed bounds on $\widetilde G(x,z)$ follow easily. Derivative bounds are also obtained in the same way. 
Finally, using the above bounds on $\widetilde Ai(k,z)$ and $\widetilde Ci(k,z)$, we get 
\begin{equation}\label{ak12-bound}\begin{aligned}  
 |\partial_x^k a_1 (x)|  &\le C\delta^{-k}(1+x)^{1/2-k/3} (1+|X|)^{k/2-1}
 \\ 
 |\partial_x^ka_2(x)| &\le    C\delta^{-k}(1+x)^{5/6-k/3} (1+|X|)^{k/2-3/2} ,
\end{aligned}
\end{equation}
upon noting that the exponents in $Ai(\cdot)$ and $Ci(\cdot)$ are cancelled out identically. 

This completes the proof of the lemma. 
\end{proof}

Similarly, we also obtain the following simple lemma. 
\begin{lemma} \label{lem-ptErrA} Let $Err_A(x,z)$ be the error kernel defined as in \eqref{def-ErrA}, and let $X = \eta(x)/\delta$ and $Z = \eta(z)/\delta$. There hold
\begin{equation}\label{ErrA-xnearz}
\begin{aligned}
| \D_z^k \D_x^\ell Err_A(x,z) |
\le 
C \delta^{1-k-\ell}(1+|z|)^{-(k+4)/3}(1+|Z|)^{(k+\ell -1)/2}
e^  {-  {\sqrt{2} \over 3} \sqrt{|Z|}|X-Z| }  .
\end{aligned}
\end{equation} 
\end{lemma} 
\begin{proof} We recall that 
$$
Err_A(x,z)=   \epsilon  (\D_z^2 \dot z^{1/2} \dot z^{-1/2} - 2\alpha^2)\dz^2G(x,z). 
$$
Thus, the lemma follows directly from the estimates on $\dz^2 G(x,z)$.
\end{proof}


\subsection{Convolution estimates}
In this section, we establish the following convolution estimates.  
\begin{lemma}  \label{lem-ConvAiry} Let $G(x,z)= \widetilde G(x,z) + E(x,z)$ be the approximate Green function of the primitive Airy equation constructed as in Lemma \ref{lem-GreenAiry}, and let $f\in X_\eta$, $\eta>0$.  
Then there is some constant $C$ so that 
\begin{equation}\label{conv-loc}\begin{aligned}
\Big\| \int_{0}^{\infty} \partial_z^k\widetilde G(x,\cdot)  f(x)    dx \Big\|_{\eta'} &\le  \frac {C \delta^{1-k} }{\eta - \eta'}  \|f\|_\eta ,
\end{aligned}\end{equation}
and \begin{equation}\label{conv-nonloc}\begin{aligned}
\Big\|\int_{0}^{\infty} \partial_z^kE(x,\cdot )  f(x)   dx \Big \|_{\eta'}&\le \frac {C}{\eta - \eta'}\|f\|_\eta ,
\end{aligned}\end{equation} for $k = 0,1,2$ and for $\eta'<\eta$. 
\end{lemma} 


\begin{proof} Without loss of generality, we assume $\|f \|_\eta =1$. First, consider the case $|z|\le 1$. Using the pointwise bounds obtained in Lemma \ref{lem-ptGreenbound}, we have 
$$\begin{aligned}
\int_{0}^{\infty} |  \widetilde G(x,z) f(x)|   dx &\le C_0 \int_0^\infty  e^{-  {\sqrt{2} \over 3} \sqrt{|Z|}|X-Z| } e^{-\eta x}\; dx \le C \delta  ,
\end{aligned}$$
upon noting that $dx = \delta \dot z^{-1}(\eta (x)) dX$ with $\dot z(\eta(x)) \approx (1+|x|)^{1/3}$. Here the growth of $\dot z(\eta(x))$ in $x$ is clearly controlled by $e^{ -\eta x } $. Similarly, since $|E(x,z)|\le C(1+x)^{4/3}$, we have 
$$\begin{aligned}
\int_z^\infty | E(x,z) f(x)|   dx &\le C \int_z^\infty (1+x)^{4/3} e^{-\eta x} dx\le C   ,
\end{aligned}$$
which proves the estimates for $|z|\le 1$.

Next, consider the case $z\ge 1$, and $k = 0,1,2$. Again using the bounds from Lemma \ref{lem-ptGreenbound} and noting that $Z = \eta(z)/\delta \approx (1+|z|)^{2/3}/\delta$ as $z$ becomes large, we obtain 
$$\begin{aligned}
\int_{0}^{\infty} |  \partial_z^k\widetilde G(x,z) f(x)|   dx
 &\le C \delta^{-k}\int_0^\infty  (1+z)^{(4-k)/3}  e^{-\eta x}e^{-  {\sqrt{2} \over 3} \sqrt{|Z|}|X-Z| } \; dx 
\\&\le C \delta^{-k}  (1+z)^{1-k/3}e^{-\eta z}   \int_0^\infty e^{-  {\sqrt{2} \over 3} \sqrt{|Z|}|X-Z| }   \delta dX 
\\&\le C\delta^{1-k}  (1+z)e^{-\eta |z|} 
\end{aligned}$$
Here again we have used the change of variable $dx = \delta \dot z^{-1} dX$ with $\dot z \approx (1+|x|)^{1/3}$. 

Let us now consider the nonlocal term $E(x,z)$, which is nonzero for $x>z$, and consider the case $z\ge 1$. We recall that 
$$|E(x,z) |\le   C(1+x)^{4/3} (1+|X|)^{-3/2} + C(1+x)^{1/3}|x-z|.$$
Let us give estimate on the integrals involving the last term in $E(x,z)$; the first term in $E(x,z)$ can be treated easily. We consider two cases: $|x-z|\le M$ and $|x-z|\ge M$ for $M = \frac{1}{\eta}\log(1+z)$. In the former case, we have 
$$\begin{aligned}
\int_{z}^{z+M} | E(x,z) f(x)|   dx &\le CM \int_z^{z+M} (1+x)^{1/3}e^{-\eta|x|} dx 
\\&\le C  (1+z)^{1/3} \log(1+z) e^{-\eta |z|} .
\end{aligned}$$
Similarly, for $x>z+M$, we have  
$$\begin{aligned}
\int_{z+M}^\infty | E(x,z) f(x)|   dx &\le C \int_{z+M}^\infty (1+x)^{4/3}e^{-\eta|x|} dx 
\\&\le C  (1+z)^{4/3} e^{-\eta M} e^{-\eta |z|}  = C  (1+z)^{1/3} e^{-\eta |z|} .
\end{aligned}$$
This completes the proof of the lemma.
\end{proof}


Similarly, we also obtain the following convolution estimate for the error kernel $Err_A(x,z)$.

\begin{lemma}  \label{lem-ErrAiry} Let $Err_A(x,z)$ be the error kernel of the primitive Airy equation defined as in Lemma \ref{lem-GreenAiry}, and let $f\in X_\eta $ for some $\eta>0$.  
Then there is some constant $C$ so that 
\begin{equation}\label{conv-ErrAiry}\begin{aligned}
\Big\|\int_{0}^{\infty}  Err_A(x,\cdot )  f(x)   dx\Big\|_\eta  &\le  C \delta    \|f\|_\eta 
\end{aligned}\end{equation}
for all $z\ge 0$. 
\end{lemma} 

\begin{proof} Again, we assume $\|f\|_\eta =1$. From the estimates in Lemma \ref{lem-ptErrA}, we in particular have $|Err_A(x,z)| \le C \delta$. Thus, the estimate is clear when $|z|\le 1$. Let us consider the case $z\ge 1$. Similarly to the estimate on $\widetilde G(x,z)$, we have
$$\begin{aligned}
\int_{0}^{\infty} | Err_A(x,z) f(x)|   dx
&\le C\delta  \int_0^\infty (1+z)^{-4/3} e^{-\eta x}e^{-  {\sqrt{2} \over 3} \sqrt{|Z|}|X-Z| } \; dx 
\\&\le C\delta e^{-\eta z}   \int_0^\infty e^{-  {\sqrt{2} \over 3} \sqrt{|Z|}|X-Z| }   \delta dX 
 \\&\le C \delta^2   e^{-\eta z}.
\end{aligned}$$
The lemma thus follows. \end{proof}


\subsection{Resolution of modified Airy equation}\label{sec-resAiry}

In this section, we shall introduce the approximate inverse of the $Airy$ operator. We recall that  $Airy(\phi) = \eps \dz^4 \phi - (U - c+ 2\epsilon \alpha^2 ) \dz^2 \phi $. Let us study the inhomogeneous Airy equation
\begin{equation}\label{Airyp-in}Airy(\phi)  =  f(z),\end{equation}
for some source $f(z)$. We introduce the approximate solution to this equation by defining 
\begin{equation}\label{eqs-Airyp-S}
AirySolver(f) := \int_{0}^{+\infty} G(x,z) f(x)  dx .
\end{equation}
Then, since the Green function $G(x,z)$ does not solve exactly the modified Airy equation (see \eqref{approxGreen}), the solution $AirySolver(f)$ does not solve it exactly either. However, there holds
\begin{equation}\label{eqs-AirySolver}
Airy(AirySolver(f)) = f + AiryErr(f)
\end{equation}
where the error operator $AiryErr(\cdot)$ is defined by
$$
AiryErr(f) : = \int_{0}^{+\infty} Err_A (x,z) f(x) dx  ,
$$
in which $Err_A (x,z)$ is the error kernel of the Airy operator, defined as in Lemma \ref{lem-GreenAiry}. In particular, from Lemma \ref{lem-ErrAiry}, we have the estimate 
\begin{equation}\label{AiryErr-iter} \| AiryErr(f)  \|_\eta  \le C\delta \|f \|_\eta,\end{equation}
for all $f\in X_\eta$. That is, $AiryErr(f)$ is indeed of order $\cO(\delta)$ in $X_\eta$. 

For the above mentioned reason, we may now define by iteration an exact solver for the modified Airy operator. Let us start with a fixed $f \in X_\eta$. Let us define
\begin{equation}\label{iter-Aphi}
\begin{aligned}
\phi_n &= - AirySolver(E_{n-1})
\\
E_n &=  - AiryErr(E_{n-1}) \end{aligned}
\end{equation}
for all $n \ge 1$, with $E_0 = f$. Let us also denote 
$$
S_n = \sum_{k=1}^n \phi_k .
$$
It  follows by induction that 
$$ Airy (S_n) = f + E_n,$$
for all $n\ge 1$. Now by \eqref{AiryErr-iter}, we have 
$$ \| E_n\|_\eta \le C \delta \|E_{n-1}\|_\eta \le (C\delta)^n \| f\|_\eta.$$ 
This proves that $E_n \to 0$ in $X_\eta$ as $n \to \infty$ since $\delta$ is small. In addition, by a view of Lemma \ref{lem-ConvAiry}, we have 
$$ \|  \phi_n \|_{\eta'} \le C  \| E_{n-1}\|_\eta \le C (C\delta)^{n-1} .$$
This shows that $\phi_n$ converges to zero in $X_{\eta'}$ for arbitrary fixed $\eta' <\eta$ as $n \to \infty$, and furthermore the series 
$$ S_n \to S_\infty$$
 in $X_{\eta'}$ as $n \to \infty$, for some $S_\infty \in X_{\eta'}$. We then denote $AirySolver_\infty(f) = S_\infty$, for each $f \in X_\eta$. In addition, we have $ Airy (S_\infty) = f,$ that is, $AirySolver_\infty(f) $ is the exact solver for the modified Airy operator. 

To summarize, we have proved the following proposition. 
\begin{prop}\label{prop-exactAiry} Let $\eta'<\eta$ be positive numbers. Assume that $\delta$ is sufficiently small. There exists an exact solver $AirySolver_\infty(\cdot)$ as a well-defined operator from $X_\eta$ to $X_{\eta'}$ so that 
$$ Airy(AirySolver_\infty (f)) = f.$$
In addition, there holds 
$$
\| AirySolver_\infty (f) \|_{\eta'} \le {C \over \eta - \eta'} \| f \|_\eta ,
$$
for some positive constant $C$. 
\end{prop}


\newpage

\section{Singularities and Airy equations}


In this section, we study the smoothing effect of the modified Airy function. Precisely, let us consider the Airy equation with a singular source:
\begin{equation}\label{Airyp-singular}Airy(\phi) = \eps \dz^4 \phi - (U - c) \dz^2 \phi =\epsilon  \D_z^4 f(z)\end{equation}
in which $f \in Y_4^\eta$, that is $f(z)$ satisfies 
\begin{equation}\label{assump-f01} |\dz^k f(z)| \le  C e^{-\eta z} ,\qquad k = 0, \cdots, 4,\end{equation}
 for $z$ away from $z_c$, and $f(z)$ behaves as $(z-z_c)\log(z-z_c)$ for $z$ near $z_c$. Precisely, we assume that 
\begin{equation}\label{assump-f02} 
|f(z)|\le C, \quad | \dz f(z) | \le C (1 + | \log (z - z_c) | ) , \quad  | \dz^k f(z) | \le C (1 + | z - z_c |^{1 - k} )
,\end{equation}
for $z$ near $z_c$ and for $k = 2,3,4$, for some constant $C$. 

We are interested in the convolution of the Green function of the Airy equation against the most singular term $\dz^4 f(z)$, or precisely the inverse of the Airy operator smoothing the singularities in the source term $\epsilon \partial_z^4 f$.

We then obtain the following crucial proposition.

\begin{prop}\label{prop-mAiry} Assume that $z_c,\delta \lesssim \alpha$. Let $AirySolver_\infty(\cdot)$ be the exact Airy solver of the $Airy(\cdot)$ operator constructed as in Proposition \ref{prop-exactAiry} and let $f\in  Y_4^\eta$. There holds the estimate: 
\begin{equation}\label{smooth-Airy}
\begin{aligned}
\Big\| AirySolver_\infty( \epsilon \D_x^4 f ) \Big\|_{X_2^{\eta'}} \le  \frac{C}{\sqrt {\eta - \eta'}} \|f\|_{Y_4^\eta}  \delta(1+|\log \delta|) (1+|z_c/\delta| ) 
\end{aligned}\end{equation}
for arbitrary $\eta' < \eta$. 
\end{prop}

We start the proof of the proposition by obtaining the same estimate for $AirySolver (\epsilon \partial_z^4 f)$ and $AiryError (\epsilon \partial_z^4 f)$. The claimed estimate for the exact solver follows the same lines as those given in Section \ref{sec-resAiry}. It thus suffices to prove the following two lemmas.

%

\begin{lemma}\label{lem-mAiry} Assume that $z_c,\delta \lesssim \alpha$. Let $G(x,z)$ be the approximated Green function to the modified Airy equation constructed as in Lemma \ref{lem-GreenAiry} and let $f\in Y_4^\eta$. There holds a convolution estimate: 
\begin{equation}\label{mphi-bound}
\begin{aligned}
\Big|(U(z)-c)^k \dz^k \int_{0}^{\infty} G(x,z) \epsilon \D_x^4 f(x) dx \Big| \le    C\|f\|_{Y_4^\eta}  \delta(1+|\log \delta|) (1+|z_c/\delta|+ \delta^{1/2} |z|^{1/3})e^{-\eta z}
\end{aligned}\end{equation}
for all $z\ge 0$, and for $k = 0,1,2$. 
\end{lemma}


Similarly, we also have the following. 
\begin{lemma}\label{lem-mErrAiry} Assume that $z_c,\delta \lesssim \alpha$. 
Let $Err_A(x,z)$ be the error defined as in Lemma \ref{lem-GreenAiry} and let $f\in  Y_4^\eta$. There holds the convolution estimate for $Err_A(x,z)$
\begin{equation}\label{mErrA-bound}
\begin{aligned}
\Big|(U(z)-c)^k \dz^k \int_{0}^{\infty}  Err_A (x,z) \epsilon\D_x^4 f(x) dx \Big| \le     C \|f\|_{ Y_4^\eta} e^{-\eta z} \delta^2 (1+|\log \delta|) 
\end{aligned}\end{equation}
for all $z\ge 0$, and for $k=0,1,2$. 
\end{lemma}


%

\begin{proof}[Proof of Lemma \ref{lem-mAiry} with $k=0$]  Let us assume that $ \|f\|_{ Y_4^\eta}  = 1$. To begin our estimates, let us recall the decomposition of $G(x,z)$ into the localized and non-localized part as 
$$G(x,z) = \widetilde G(x,z) + E(x,z),$$
where $\widetilde G(x,z) $ and $E(x,z) $ satisfy the pointwise bounds in Lemma \ref{lem-ptGreenbound}. In addition, we recall that $\epsilon \D_x^j G_{2,a}(X,Z)$ and so $\epsilon \D_x^j G(x,z)$ are continuous across $x=z$ for $j=0,1,2$. Using the continuity, we can integrate by parts to get  
\begin{equation}\label{mphi-integral}\begin{aligned}
\phi (z)&=  - \epsilon \int_{0}^{\infty}\D_x^3  (\widetilde G + E) (x,z) \D_x f(x) \; dx 
+ \mathcal{B}_0(z)\\
&= I_\ell (z) + I_e(z) +  \mathcal{B}_0(z)
\end{aligned}
\end{equation}
Here, $I_\ell(z) $ and $I_e(z)$ denote the corresponding integral that involves $\widetilde G(x,z)$ and $E(x,z)$ respectively, and $\mathcal{B}_0(z)$ is introduced to collect the boundary terms at $x=0$ and is defined by 
\begin{equation}\label{def-mBdry}
\mathcal{B}_0(z): = - \epsilon \sum_{k=0}^2 (-1)^k \D_x^k G(x,z) \D^{3-k}_x(f(x))\vert{_{x=0}} .
\end{equation}
By a view of the definition of $E(x,z)$, we further denote 
$$\begin{aligned}  
I_{e,1} (z) : &=  i \delta^2 \pi  \int_{z}^{\infty} \D_x^3(\dot x^{3/2} a_1 (x) (z-x) ) \D_x f(x) \; dx,
\\ I_{e,2} (z) :& =   i \delta^3 \pi \int_{z}^{\infty} \D_x^3 ( \dot x^{3/2} a_2(x)  ) \D_x f(x) \; dx  
\end{aligned}$$
We have $I_e(z) = I_{e,1}(z) + I_{e,2}(z)$. 

\bigskip
\noindent
{\bf Estimate for the integral $I_\ell (z)$.}  Using the bound \eqref{mG-xnearz} on the localized part of the Green function, we can give bounds on the integral term $I_\ell $ in \eqref{mphi-integral}. Consider the case $|z-z_c|\le \delta$. In this case, we note that $\eta'(z) \approx \dot z(\eta(z)) \approx 1$.  By splitting the integral into two cases according to the estimate \eqref{mG-xnearz}, we get
$$\begin{aligned}
|I_\ell (z)| &=\Big| \epsilon \int_{0}^{\infty} \D_x^3  \widetilde G(x,z) \D_x f(x) \; dx \Big| 
\\& \le  \epsilon \int_{\{|x-z_c|\le \delta\}} |\D_x^3  \widetilde G(x,z) \D_x f(x)| \; dx +  \epsilon \int_{\{|x-z_c|\ge \delta\}} |\D_x^3  \widetilde G(x,z) \D_x f(x)| \; dx,
 \end{aligned}$$ 
in which since $\epsilon \D_x^3 \widetilde G(x,z)$ is uniformly bounded, the first integral on the right is bounded by 
 $$ C \int_{\{|x-z_c|\le \delta\}} | \D_x f(x)| \; dx  \le C   \int_{\{|x-z_c|\le \delta\}} (1+|\log (x-z_c)|) \; dx \le C \delta (1+|\log \delta|) .$$ 
For the second integral on the right, we note that in this case since $X$ and $Z$ are away from each other, there holds $e^  {-  {\sqrt{2} \over 3} \sqrt{|Z|}|X-Z| } \le C e^{ - \frac 16 |X|^{3/2} }e^{ - \frac 16 |Z|^{3/2} }$. We get 
$$\begin{aligned}
\epsilon \int_{\{|x-z_c|\ge \delta\}}  |\D_x^3  \widetilde G(x,z) \D_x f(x)| \; dx 
& \le C   \int_{\{|x-z_c|\ge \delta\}}  e^{ - \frac 16 |X|^{3/2} } e^{-\eta  x } (1+|\log (x-z_c)|)\; dx 
\\ & \le C   (1+|\log\delta|)\int_\RR   e^{ - \frac 16 |X|^{3/2} }  \; dx 
\\& \le C   \delta (1+|\log\delta|),
\end{aligned}$$
in which the second-to-last inequality was due to the crucial change of variable $X = \delta^{-1}\eta(x)$ and so $dx = \delta \dot z (\eta (x)) dX$ with $|\dot z(\eta (x))| \le C(1+|x|)^{1/3}$.

Let us now consider the case $|z-z_c|\ge \delta$. Here we note that as $z\to \infty$, $Z = \delta^{-1}\eta(z)$ also tends to infinity since $|\eta(z)| \approx (1+|z|)^{2/3}$ as $z$ is sufficiently large. We again split the integral in $x$ into two parts $|x-z_c|\le \delta $ and $|x-z_c|\ge \delta$. For the integral over $\{|x-z_c|\le \delta\}$, as above, with $X$ and $Z$ being away from each other, we get 
$$\begin{aligned}
\epsilon \int_{\{|x-z_c|\le \delta\}} |\D_x^3  \widetilde G(x,z) \D_x f(x)| \; dx 
& \le C   
e^{ - \frac 16 |Z|^{3/2}}  \int_{\{|x-z_c|\le \delta\}}  (1+|\log (x-z_c)|)\; dx 
\\&  \le C   e^{-\eta z}\delta (1+|\log \delta|).
\end{aligned}$$
Here the exponential decay in $z$ was due to the decay term $e^{ - \frac 16 |Z|^{3/2}}$ with $Z \approx (1+z)^{2/3}$. 
Next, for the integral over $\{ |x-z_c|\ge \delta\}$, we use  the bound \eqref{mG-xnearz} and the assumption $|\D_x f(x)| \le C   e^{-\eta  x } (1+|\log \delta|)$ to get 
$$\begin{aligned}
\epsilon \int_{\{|x-z_c|\ge \delta\}} & |\D_x^3  \widetilde G(x,z) \D_x f(x)| \; dx  
\\&\le C    (1+|\log \delta|) (1+z)^{1/3}  \int e^{-\eta x}   e^  {- \sqrt{2|Z|}|X-Z| /3} \; dx 
\\&\le C (1+|\log \delta|) \delta (1+z)^{1/3}  e^{-\eta z }  |Z|^{-1/2}
 \end{aligned}$$ 
If $z\le 1$, the above is clearly bounded by $C (1+|\log \delta|) \delta$. Consider the case $z\ge 1$. We note that $|Z| \gtrsim |z|^{2/3} / \delta$. This implies that $(1+z)^{1/3}|Z|^{-1/2} \lesssim 1$ and so the above integral is again bounded by $C (1+|\log \delta|) \delta e^{-\eta z } $.

 Therefore in all cases,  we have $|I_\ell (z)|\le C  e^{-\eta z} \delta (1+|\log \delta|)$ or equivalently, 
\begin{equation}  \Big| \int_{0}^{\infty}\epsilon \D_x^3  \widetilde G (x,z) \D_x f(x) \; dx \Big| 
\le C  e^{-\eta z} \delta (1+|\log \delta|)\end{equation}
 for all $z\ge 0$. 
 
 \bigskip
\noindent
{\bf Estimate for $I_{e,2}$.} Again, we consider several cases depending on the size of $z$. For $z$ away from the critical and boundary layer (and so is $x$): $z\ge |z_c| + \delta$, we apply integration by parts to get  
$$\begin{aligned}
 I_{e,2} (z) &=  i \delta^3 \pi \int_{z}^{\infty} \D_x^3 ( \dot x^{3/2} a_2(x)  ) \D_x f(x) \; dx  
 \\& =  -i \delta^3 \pi \int_{z}^{\infty} \D_x^2 (\dot x^{3/2}a_2(x)) \D^2_x f(x) \; dx  -  i \delta^3 \pi \D_x^2 (\dot x^{3/2}a_2(x)) \D_x f(x)\vert_{x=z} .
 \end{aligned}$$
Here for convenience, we recall the bound \eqref{ak12-bound} on $a_2(x)$: 
\begin{equation}\label{bound-ma2}
 |\partial_x^ka_2(x)| \le    C\delta^{-k}(1+x)^{5/6-k/3} (1+|X|)^{k/2-3/2} .
\end{equation}
Now by using this bound and the fact that $|Z| \gtrsim |z|^{2/3}/\delta$, the boundary term is clearly bounded by 
$$\begin{aligned}
 C\delta (1+|z|)^{2/3} & (1+|Z|)^{-1/2} (1+|\log (z-z_c)|) e^{-\eta z}  
 \\&  \le C  e^{-\eta z} \delta (1+|\log \delta|) (1+\delta^{1/2}|z|^{1/3})\end{aligned}$$
whereas the integral term is estimated by 
$$\begin{aligned}
\Big|\delta^3 \pi &\int_{z}^{\infty} \D_x^2 (\dot x^{3/2}a_2(x)) \D^2_x f(x) \; dx \Big| 
\\&\le C \delta  \int_{z}^{\infty} (1+x)^{2/3}|X|^{-1/2} |x-z_c|^{-1}   e^{-\eta x} \; dx
\\&\le C \delta (1+z)^{2/3} |Z|^{-1/2}   e^{-\eta z} (1+|\log \delta|)
\\&\le C   \delta(1+|\log \delta|) (1+\delta^{1/2} |z|^{1/3})e^{-\eta z} .
\end{aligned}$$
Thus we have 
\begin{equation}\label{mIe2-est}\Big|I_{e,2} (z) \Big| \le  C   \delta(1+|\log \delta|) (1+\delta^{1/2} |z|^{1/3})e^{-\eta z} , \end{equation}
for all $z\ge |z_c|+\delta$. 

Next, for $z\le |z_c| + \delta$, we write the integral $I_{e,2} (z)$ into 
$$ \delta^3 \int_{\{|x-z_c| \ge \delta\}} \D_x^3 (\dot x^{3/2}a_2(x) )\D_x f(x) \; dx  +  \delta^3 \int_{\{|x-z_c|\le \delta\}}\D_x^3 (\dot x^{3/2}a_2(x)) \D_x f(x) \; dx , $$ 
where the first integral can be estimated similarly as done in \eqref{mIe2-est}. For the last integral, using \eqref{bound-ma2} for bounded $X$ yields 
$$\begin{aligned}
\Big|\delta^3 \int_{\{|x-z_c|\le \delta\}}\D_x^3 (\dot x^{3/2}a_2(X)) \D_x f(x) \; dx  \Big| 
&\le C   \int_{\{|x-z_c|\le \delta\}} (1+|\log (x-z_c)|)\; dx
\\& \le C   \delta (1+|\log \delta|) .\end{aligned}$$

Thus, we have shown that 
\begin{equation}\label{mIe2-bound} \Big| I_{e,2} (z) \Big| \le C   \delta(1+|\log \delta|) (1+\delta^{1/2} |z|^{1/3})e^{-\eta z},\end{equation}
for all $z\ge 0$. 

\bigskip
\noindent
{\bf Estimate for $I_{e,1}$.} Following the above estimates, we can now consider the integral 
 $$\begin{aligned}  
I_{e,1} (z) =  i \delta^2 \pi  \int_{z}^{\infty} \D_x^3(\dot x^{3/2} a_1 (x) (z-x) ) \D_x f(x) \; dx,
\end{aligned}$$
Let us recall the bound \eqref{ak12-bound} on $a_1(x)$: 
\begin{equation}\label{bound-ma1} |\partial_x^k a_1 (x)|  \le C\delta^{-k}(1+x)^{1/2-k/3} (1+|X|)^{k/2-1}.\end{equation}
To estimate the integral $I_{e,1}(z)$, we again divide the integral into several cases. First, consider the case $z\ge |z_c| + \delta$. Since in this case $x$ is away from the critical layer, we can apply integration by parts three times to get  
$$\begin{aligned} 
I_{e,1} (z)&=  -i \delta^2 \pi\int_{z}^{\infty} \D_x^2(\dot x^{3/2} a_1 (x) (z-x)) \D_x^2 f(x) \; dx  - i \delta^2 \pi\D_x^2(\dot x^{3/2} a_1 (x) (z-x)) \D_x f(x) \vert_{x=z} 
\\  & = - i \delta^2 \pi \int_{z}^{\infty}  \dot x^{3/2} a_1 (x) (z-x) \D_x^4 f(x) \; dx  
\\&\quad  + i \delta^2 \pi \Big( \D_x(\dot x^{3/2} a_1 (x) (z-x)) \D_x^2 f(x)  -\D_x^2(\dot x^{3/2} a_1 (x) (z-x)) \D_x f(x) \Big) \vert_{x=z} 
 \end{aligned}$$
in which the boundary terms are bounded by $   e^{-\eta |z|}\delta (1+|\log \delta|)$ times 
$$\begin{aligned}
 (1+z)^{1/6}&\Big[ \delta  (1+z)^{5/6} |Z|^{-1} |z-z_c|^{-1} + \delta (1+z)^{-1/6} |Z|^{-1} + (1+z)^{1/2} |Z|^{-1/2}\Big]  
 \\& \le C (1+\delta^{1/2}z^{1/3}).
  \end{aligned}$$ 
  
Similarly, we consider the integral term in $I_{e,1}$. Let $M = \frac 1 \eta \log(1+z)$. By \eqref{bound-ma1}, we have 
$$\begin{aligned}
\Big| \delta^2 \pi &\int_{z}^{\infty}  \dot x^{3/2} a_1 (x) (z-x) \D_x^4 f(x) \; dx   \Big| 
\\& \le  C    \int_z^\infty \delta^2 (1+x)(1+|X|)^{-1}|x-z| |x-z_c|^{-3}  e^{-\eta  x }\; dx 
\\& \le  C    (1+z)^{1/3}\Big[ M +(1+z) e^{-\eta M}  \Big]  e^{-\eta z} \int_{\{|x-z_c|\ge \delta\}} \delta^3 |x-z_c|^{-3} \; dx 
\\& \le C   (1+z)^{1/3}\log(1+z) e^{-\eta z}\delta . 
\end{aligned}$$
Hence, we obtain the desired uniform bound $I_{e,1}(z)$ for $z \ge |z_c| + \delta $.

Next, consider the case $|z-z_c|\le \delta$ in which $Z$ is bounded. We write 
$$ I_{e,1}(z) =    i \delta^2 \pi \Big[ \int_{\{|x-z_c|\ge \delta\}} + \int_{\{|x-z_c|\le\delta\}} \Big]  \D_x^3(\dot x^{3/2}a_1(x)(z-x)) \D_x f(x) \; dx   . $$ The first integral on the right can be estimated similarly as above, using integration by parts. For the second integral, we use the bound \eqref{bound-ma1} for bounded $X$ to get 
$$\begin{aligned}
\Big|\delta^2\int_{\{|x-z_c|\le\delta\}}\D_x^3(\dot x^{3/2}a_1(x)(z-x)) \D_x f(x)  \; dx \Big| &\le C \int_{\{|x-z_c|\le\delta\}} (1+|\log (x-z_c)|)\; dx
\\&  \le  C   \delta (1+|\log \delta|) .
\end{aligned}$$ 

Finally, we consider the case $0\le z\le |z_c|-\delta$. This is the case when $\delta \ll |z_c|$, that is the critical layer is away from the boundary layer. In this case the linear growth in $Z$ becomes significant: $|Z| \lesssim (1+|z_c|/\delta)$. Thus, following the above analysis, we obtain 
\begin{equation}\label{large-zc}  I_{e,1}(z) \le  C   (1+|z_c/\delta|)\delta (1+|\log \delta|) .\end{equation}
 
The estimate for $I_{e,1}(z)$ thus follows for all $z\ge 0$.

 \bigskip
 
 \noindent
 {\bf Estimate for the boundary term $\mathcal{B}_0(z)$.} It remains to give estimates on
$$\mathcal{B}_0(z) = - \epsilon \sum_{k=0}^2 (-1)^k \D_x^k G(x,z) \D^{3-k}_x(f(x))\vert{_{x=0}} .$$
We note that there is no linear term $E(x,z)$ at the boundary $x=0$ since $z\ge 0$. Using the bound \eqref{mG-xnearz} for $x = 0$, we get   
$$ \begin{aligned}
|\epsilon \widetilde G(x,z) \D^3_x(f(x))\vert{_{x=0}}  &\le C    \delta^3  (1+ |z_c|^{-2})  e^{-\frac 23 |Z|^{3/2}}
\\
|\epsilon \D_x\widetilde G(x,z) \D^2_x(f(x))\vert{_{x=0}}  &\le C    \delta^2  (1+|z_c|^{-1}) e^{-\frac 23 |Z|^{3/2}}
\\
|\epsilon \D^2_x\widetilde G(x,z) \D_xf(x)\vert{_{x=0}} &\le C    \delta (1+|\log z_c|)e^{-\frac 23 |Z|^{3/2}}.
\end{aligned}$$
This together with the assumption that $\delta \lesssim z_c $ then yields 
\begin{equation}\label{B-est01}| \mathcal{B}_0(z) | \le C    \delta(1+|\log \delta|) e^{-\eta z}.\end{equation}

Combining all the estimates above yields the lemma for $k=0$. 
\end{proof}

\begin{proof}[Proof of Lemma \ref{lem-mAiry} with $k>0$] We now prove the lemma for the case $k=2$; the case $k=1$ follows similarly. We consider the integral 
$$ \epsilon \int_{0}^{\infty} (U(z) - c)^2 \dz^2 \widetilde G(x,z) \D_x^4 f(x) dx  = I_1(z) + I_2(z),$$ with $I_1(z)$ and $I_2(z)$ denoting the 
integration over $\{ |x-z_c|\le \delta\}$ and $\{|x-z_c|\ge \delta\}$, respectively.  
Note that $(U(z)-c)\dot z^{2} =  U'(z_c) \eta(z)$ and recall that $Z = \eta(z)/\delta$ by definition. For the second integral $I_2(z)$, by using \eqref{assump-f01}, \eqref{assump-f02}, and the bounds on the Green function for $x$ away from $z$ and for $x$ near $z$, it follows easily that 
$$\begin{aligned} |I_2(z)| &\le C \Big[ \delta e^{-\eta z} \int_{\{|x-z_c|\ge \delta\}}  (1+|Z|)^{1/2} e^{-\frac 23 \sqrt{|Z|} |X-Z|} (1+|x-z_c|^{-1})\; dx
\\&\quad + \epsilon e^{-\frac 16 |Z|^{3/2}} \int_{\{|x-z_c|\ge \delta\}}e^{-\frac 16 |X|^{3/2}} (1+|x-z_c|^{-3}) e^{-\eta x}\; dx   \Big] .\end{aligned}$$
Using $|x-z_c|\ge \delta$ in these integrals and making a change of variable $X = \eta(x)/\delta$ to gain an extra factor of $\delta$, we obtain 
$$\begin{aligned} |I_2(z)| &\le C \delta  \Big[(1+z) e^{-\eta z} \int_{\RR}  (1+|Z|)^{1/2} e^{-\frac 23 \sqrt{|Z|} |X-Z|} \; dX 
+ e^{-\frac 18 |Z|^{3/2}} \int_{\RR}e^{-\frac 16 |X|^{3/2}}\; dX  \Big] ,\end{aligned}$$
which is clearly bounded by $C \delta (1+z) e^{-\eta z}$. It remains to give the estimate on $I_1(z)$ over the region: $|x-z_c|\le \delta$. In this case, we take integration by parts three times. Leaving the boundary terms untreated for a moment, let us consider the integral term
$$\epsilon \int_{\{|x-z_c|\le \delta\}} (U(z) - c)^2 \dz^2 \dx^3 \widetilde G(x,z) \D_x f(x) dx.$$
We note that the twice $z$-derivative causes a large factor $\delta^{-2}$ which combines with $(U-c)^2$ to give a term of order $|Z|^2$. Similarly, the small factor of $\epsilon$ cancels out with $\delta^{-3}$ that comes from the third $x$-derivative. The integral is therefore easily bounded by 
$$\begin{aligned}  C  \Big[ &e^{-\eta z} \int_{\{|x-z_c|\le \delta\}}  e^{-\frac 23 \sqrt{|Z|} |X-Z|} (1+|\log (x-z_c)|)\; dx
\\&\quad + e^{-\frac 16 |Z|^{3/2}} \int_{\{|x-z_c|\le \delta\}}e^{-\frac 16 |X|^{3/2}} (1+|\log(x-z_c)|)\; dx   \Big] 
\\
&\le  C  \Big[ e^{-\eta z} + e^{-\frac 18 |Z|^{3/2}}   \Big] \int_{\{|x-z_c|\le \delta\}} (1+|\log (x-z_c)|)\; dx 
\\
&\le  C e^{-\eta z} \delta (1+|\log \I c|).
\end{aligned}$$
Finally, the boundary terms can be treated, following the above analysis and that done in the case $k=0$; see \eqref{B-est01}.   
This completes the proof of the lemma.\end{proof}

\begin{proof}[Proof of Lemma \ref{lem-mErrAiry} ] The proof follows similarly, but more straightforwardly, the above proof for the localized part of the Green function, upon recalling that 
$$ Err_A(x,z) = \epsilon  (\D_z^2 \dot z^{1/2} \dot z^{-1/2} - 2\alpha^2)\dz^2G(x,z).$$
We skip the details. \end{proof}

%
%
%
%
%
%
%

\newpage 

\section{Construction of slow Orr-Sommerfeld modes}\label{sec-construction-phi1}


In this section, we iteratively construct two exact slow-decaying and -growing solutions $\phi_{1,2}$. 
Precisely, we obtain the following proposition whose proof will be given at the end of the section, yielding an exact solution to the Orr-Sommerfeld equations, starting from the exact solution to the Rayleigh equation.   

\begin{prop}\label{prop-construction-exactphi1} Let $\phi_{Ray}\in X_\alpha$ be an exact solution to the Rayleigh equation: $Ray_\alpha (\phi_{Ray}) = f$, with $f \in X_\eta$, for $\eta>0$. For sufficiently small $\alpha, \epsilon$, there exists an exact solution $\phi_s(z)$ in $X_\alpha$ which solves the Orr-Sommerfeld equations
$$Orr(\phi_s) = f,$$
so that $\phi_s$ is close to $\phi_{Ray}$ in $X_2^\eta$. Precisely, we have 
$$ \| \phi_s - \phi_{Ray}\|_{X_2^\eta} \le C \delta(1+|\log \delta|) (1+|z_c/\delta| ),$$
for some positive constant $C$ independent of $\alpha, \epsilon$.   
\end{prop}


For instance, if we start our construction with the exact Rayleigh solutions $\phi_{Ray,\pm}$, which were constructed from Lemma \ref{lem-exactphija}. Proposition \ref{prop-construction-exactphi1} yields existence of two exact solutions $\phi_{s,\pm}$ to the homogenous Orr-Sommerfeld equation.

Next, we obtain the following lemma. 
\begin{lemma}\label{lem-analytic01} 
The slow modes $\phi_s$ constructed in Proposition \ref{prop-construction-exactphi1} depend analytically in $c$, for $\I c>0$. 
\end{lemma}
\begin{proof} The proof is straightforward since the only ``singularities'' are of the forms: $\log(U-c)$, $1/(U-c)$, $1/(U-c)^2$, and $1/(U-c)^3$, which are of course analytic in $c$ when $\I c>0$.  
\end{proof}

\begin{remark}  \textup{It can be shown that the approximated solution $\phi_{N}$ can be extended $C^\gamma$-H\"older continuously on the axis $\{\I c =0\}$, for $0\le \gamma<1$. 
}\end{remark}

\subsection{Principle of the construction}


The proof of Proposition \ref{prop-construction-exactphi1} follows at once from the following proposition, providing approximate solutions to the Orr-Sommerfeld equations.

\begin{proposition}\label{prop-construction-phi1}
Let $N$ be arbitrarily large. Under the same assumptions of Proposition \ref{prop-construction-exactphi1}, there exists a function $\phi_{N} \in X_\alpha$ such that $\phi_{N}$
approximately solves the Orr-Sommerfeld equation in the sense that 
\begin{equation}\label{def-phiNa} 
Orr(\phi_{N})(z) = f + O_N(z),
\end{equation}
with the error $O_N(z)$ satisfying 
$$ 
\| O_N\|_{X_2^\eta}\le \Big[C \delta(1+|\log \delta|) (1+|z_c/\delta| )\Big]^N  .
$$
\end{proposition}

As already discussed in the introduction, we start from the Rayleigh solution $\phi_{Ray}$ so that 
$$
Ray_\alpha (\phi_{Ray}) = f.
$$
By definition, we have
\begin{equation}\label{Orr-1stapp} Orr(\phi_{Ray}) = f - Diff (\phi_{Ray}).\end{equation}
Next, we introduce
$$
A_{s}:= \chi  Diff (\phi_{Ray}) , \qquad I_0: = (1-\chi ) Diff (\phi_{Ray}) 
$$
in which $\chi(z)$ is a smooth cut-off function such that $\chi = 1$ on $[0,1]$ and zero on $[2,\infty)$. We also let 
$$B_{s} :=  AirySolver_\infty(A_{s}) , \qquad J_0 : = \partial_z^{-2}\mathcal{A}^{-1}_{a,\infty} (I_0) (z)
$$ in which $\partial_z^{-1} = -\int_z^\infty$. We then define
\begin{equation}\label{def-phi1-s0}
\phi_{1} := \phi_{Ray} + B_{s} + J_0 .
\end{equation}
We note that by the identities \eqref{key-ids}, $Airy(J_0) = \mathcal{A}_\mathrm{a} (\partial_z^2 J_0)$, and the fact that  $\partial_z^2 J_0 =  \mathcal{A}^{-1}_{a,\infty} (I_0)$, there hold
$$\begin{aligned}
 Orr(B_{s}) &=  A_{s} + Reg (AirySolver_\infty(A_{s}))  
 \\
 Orr(J_0) & = I_0 + Reg (J_0).
 \end{aligned}$$
Putting these together with \eqref{Orr-1stapp}, we get 
 $$
Orr( \phi_{1} ) =  f + O_1, \qquad O_1:=  Reg (AirySolver_\infty(A_{s}))  + Reg (J_0) ,
$$ with $Reg(\phi): =- (\eps \alpha^4 + U'' + \alpha^2 (U-c) )\phi$. 

Inductively, let us assume that we have constructed $\phi_{N}$ so that $$
 Orr(\phi_{N})  = f +  O_N,
$$ 
with an error $O_N$ which is sufficiently small in $X_\eta$. We then improve the error term by constructing a new approximate solution $\phi_{1,N+1}$ so that it solves the Orr-Sommerfeld equations with a better error in $X_\eta$. To do so, we first solve the Rayleigh equation by introducing
$$
\psi_{N} := -  RaySolver_{\alpha,\infty} \Bigl( O_N  \Bigr) .
$$ Observe that by a view of \eqref{key-ids} and \eqref{eqs-RaySolver}
\begin{equation}\label{Orr-N}
Orr (\phi_{N} + \psi_{N})  =f  - Diff ( RaySolver_{\alpha,\infty}( O_N )).  
\end{equation}
As in the previous step, we introduce
$$
A_{s,N}:= \chi Diff ( RaySolver_{\alpha,\infty}( O_N )), \qquad I_{N}: = (1-\chi )Diff ( RaySolver_{\alpha,\infty}( O_N )) 
$$
in which $\chi(z)$ is a smooth cut-off function such that $\chi = 1$ on $[0,1]$ and zero on $[2,\infty)$. We also let 
$$B_{s,N} :=  AirySolver_\infty(A_{s,N}) , \qquad J_N : = \partial_z^{-2}\mathcal{A}^{-1}_{a,\infty} (I_N) (z)
$$ in which $\partial_z^{-1} = -\int_z^\infty$. We then define
\begin{equation}\label{def-phi1N}
\phi_{1,N+1} := \phi_{N} + \psi_N + B_{s,N} + J_N .
\end{equation}
We note that by the identities \eqref{key-ids}, $Airy(J_N) = \mathcal{A}_\mathrm{a} (\partial_z^2 J_N)$, and the fact that  $\partial_z^2 J_N =  \mathcal{A}^{-1}_{a,\infty} (I_N)$, there hold
$$\begin{aligned}
 Orr(B_{s,N}) &=  A_{s,N} + Reg (AirySolver_\infty(A_{s,N}))  
 \\
 Orr(J_N) & = I_N + Reg (J_N).
 \end{aligned}$$

Putting these together with \eqref{Orr-N}, we get 
 $$
Orr( \phi_{1,N+1} ) =  f +   Reg (AirySolver_\infty(A_{s,N}))  + Reg (J_N) .
$$ with $Reg(\phi): =- (\eps \alpha^4 + U'' + \alpha^2 (U-c) )\phi$. To ensure the convergence, let us introduce the iterating operator 
\begin{equation}\label{def-Iter}
\begin{aligned}
Iter(g) :&=  Reg (AirySolver_\infty(A_{s}(g)))  + Reg \Big(\partial_z^{-2}\mathcal{A}^{-1}_{a,\infty} (I(g)) \Big) \end{aligned}
\end{equation}
in which $A_{s}(g): = \chi Diff ( RaySolver_{\alpha,\infty}( g))$ and $ I(g): = (1-\chi )Diff ( RaySolver_{\alpha,\infty}( g))$. Then 
$$
Orr(\phi_{N+1}) = f + O_{N+1}, \qquad O_{N+1} := Iter(O_N) .
$$
We then inductively iterate this procedure to get an accurate approximation to $\phi_1$. We shall prove the following key lemma which gives sufficient estimates on the $Iter$ operator and would therefore complete the proof of Proposition \ref{prop-construction-phi1}.  

\begin{lemma}\label{lem-keyIter} For $g \in X_2^\eta$, the $Iter(\cdot)$ operator defined as in \eqref{def-Iter} is a well-defined map from $X_2^\eta$ to $X_2^\eta$. Furthermore, there holds 
\begin{equation}\label{est-keyIter} \| Iter(g)\|_{X_2^\eta} \le C \delta(1+|\log \delta|) (1+|z_c/\delta| )  \|g\|_{X_2^\eta},\end{equation}
for some universal constant $C$. 
\end{lemma}

\begin{proof} Let $g \in X_2^\eta$. We give estimates on each term in $Iter(f)$. We recall that $Diff(h) = -\eps (\dz^2 - \alpha^2)^2 h$. In addition, from the identity $Ray_\alpha (RaySolver_{\alpha , \infty}(g)) = g$, we have
$$(\partial_z^2 - \alpha^2 ) RaySolver_{\alpha , \infty}(g)    =  \frac{U'' RaySolver_{\alpha , \infty}(g)}{U-c}  + \frac{g}{U-c} .$$
Thus, by a view of Proposition \ref{prop-exactRayS} and the fact that $U''$ decays exponentially, we have 
$$\begin{aligned}
| (\partial_z^2 - \alpha^2 )^2 RaySolver_{\alpha , \infty}(g) (z) |   
&\le C e^{-\eta z} \|g\|_{X_2^\eta} ,\end{aligned}$$
for all $z\ge 1$. We note that since we consider $z\ge 1$, there is no singularity coming from the critical layer: $U(z_c)=c $. We note also that that on the right hand side, the derivatives of $f$ appear at most at the second order. This proves 
$$ \|I(g)\|_\eta = \|(1-\chi ) Diff ( RaySolver_{\alpha,\infty}(g)) \|_\eta \le C \epsilon \|g\|_{X_2^\eta}.$$
with $I(g) = (1-\chi ) Diff ( RaySolver_{\alpha,\infty}(g))$ as defined in the $Iter(\cdot)$ operator. 

Now, by Proposition \ref{prop-exactAiry-cl}, we have 
$$
\| \mathcal{A}^{-1}_{a,\infty} (I(g)) \|_{\eta} \le C \delta^{-1} \| I(g) \|_\eta \le C \delta^2 \|g\|_{X_2^\eta}.
$$
Clearly, for $g \in X_\eta$, we have $\| \partial^{-1}_z g \|_\eta \le C \|g\|_\eta$. This yields 
$$\| \partial_z^{-2}\mathcal{A}^{-1}_{a,\infty} (I(g))\|_{X_2^\eta}\le C \delta \|g\|_{X_2^\eta} ,$$
which proves at once 
\begin{equation}\label{far-Iter}
\| \partial_z^{-2}\mathcal{A}^{-1}_{a,\infty} (I(g)) \|_{X_2^\eta} \le C \delta^2 \|g\|_{X_2^\eta}.   
 \end{equation}
We remark that there is no loss of derivatives in the above estimate.

$$AirySolver_\infty(A_{s}(g))$$

Next, it remains to give estimates for the terms involving the critical layer. We recall that 
$$A_{s}(g) = \chi Diff ( RaySolver_{\alpha,\infty}( g)),$$ 
which clearly belongs to $X_{\eta_1,4}$, for arbitrary $\eta_1>0$. The reason for this is precisely due to the cut-off function $\chi$ which vanishes identically for $z\ge 2$. The singularity is up to order $(z-z_c)^{-3}$ due to the $z\log z$ singularity in $RaySolver_{\alpha,\infty}(\cdot)$.   

By a view of Proposition \ref{prop-exactRayS}, we have
\begin{equation}\label{bound-grecall}\| \chi(z) RaySolver_{\alpha , \infty}(g) \|_{Y_4^{\eta_1}}\le C \|g\|_{X_2^\eta}.\end{equation}
In addition, by applying Proposition \ref{prop-mAiry}, we get 
$$\Big\| AirySolver_\infty( \chi Diff (h) ) \Big\|_{X_2^\eta} \le  \frac{C}{\sqrt {\eta_1 - \eta}} \|h\|_{Y_4^{\eta_1}}  \delta(1+|\log \delta|) (1+|z_c/\delta| ) .$$
By taking $\eta_1 = 1+\eta$, this together with \eqref{bound-grecall} yields 
\begin{equation}\label{near-Iter}
\Big\| AirySolver_\infty(A_s(g) ) \Big\|_{X_2^\eta} \le  C \delta(1+|\log \delta|) (1+|z_c/\delta| ) \|g\|_{X_2^\eta}.\end{equation}

It is now straightforward to conclude Lemma \ref{lem-keyIter} simply by combining \eqref{far-Iter} and \eqref{near-Iter}, upon recalling that $Reg(\phi): =- (\eps \alpha^4 + U'' + \alpha^2 (U-c) )\phi$. 
\end{proof}


\subsection{First order expansion of the slow-decaying mode $\phi_s$} 


In this paragraph we explicitly compute the boundary contribution of the first terms in the expansion of the slow Orr-Sommerfeld modes. We recall that the leading term from \eqref{def-phi1-s0} reads
\begin{equation}\label{def-phi1}\begin{aligned}
 \phi_1(z;c) &= \phi_{Ray}(z;c) +  AirySolver_\infty(A_{s})(z) + \partial_z^{-2}\mathcal{A}^{-1}_{a,\infty} (I_0) (z) 
 \end{aligned}\end{equation}
in which $A_{s}:= \chi  Diff (\phi_{Ray})$, $ I_0: = (1-\chi ) Diff (\phi_{Ray}) 
$, and $\phi_{Ray}(z;c) = \phi_{Ray,-}(z)$ as constructed from Lemma \ref{lem-exactphija}. There, we recall that 
$$\phi_{Ray,-} (z)= e^{-\alpha z} (U-c + \cO(\alpha)).$$
Thus, together with Proposition \ref{prop-exactAiry-cl}, 
$$\| \partial_z^{-2}\mathcal{A}^{-1}_{a,\infty} (I_0)\|_\eta \le C \|\mathcal{A}^{-1}_{a,\infty} (I_0)\|_\eta \le C \delta^{-1}\|I_0\|_\eta \le C \delta^2.$$

 Next, with $A_{s}= \chi  Diff (\phi_{Ray})$, we can write 
 $$A_{s}= \chi  Diff ( e^{-\alpha z}(U-c)) + \chi Diff (\cO(\alpha)),$$
 in which the first term consists of no singularity, and of order $\cO(\epsilon)$. We only need to apply the smoothing-singularity lemma to the last term in $A_s$.  Propositions \ref{prop-mAiry} and \ref{prop-exactAiry} thus yield
 $$\| AirySolver_\infty(A_{s})\|_\eta \le C \epsilon + C\alpha \delta(1+|\log \delta|) (1+|z_c/\delta| ).$$

This proves that 
\begin{equation}\label{diff-phija} \| \phi_1(\cdot;c) -  \phi_{Ray,-}\|_\eta \le   C \delta^2 + C\alpha \delta(1+|\log \delta|) (1+|z_c/\delta| ). \end{equation}

 
In this section, we will prove the following lemma.

\begin{lemma}\label{lem-ratiophi1} Let $\phi_1$ be defined as in \eqref{def-phi1}, and let $U_0' \not =0$. For small $z_c, \alpha, \delta$, such that $\delta \lesssim \alpha$ and $z_c\approx \alpha$, there hold
\begin{equation}\label{keyratio-phi1}
\begin{aligned}
 \frac{\phi_1(0;c)}{\partial_z\phi_1(0;c)} &= \frac{1}{U_0'}\Big[ U_0 - c + \alpha \frac{(U_+-U_0)^2}{U_0'} + \cO(\alpha^2\log \alpha) \Big] ,\quad
\\
\I  \frac{\phi_1(0;c)}{\partial_z\phi_1(0;c)} &= \frac{-\I c}{U'_0} \Big[ 1+2\alpha \frac{U_+-U_0}{U_0'}  + \cO(\alpha^2\log \alpha)  \Big] + \cO(\alpha) \delta |\log \delta| (1+|z_c/\delta| ). 
  \end{aligned}
\end{equation}
Here, $\cO(\cdot)$ is to denote the bound in $L^\infty$ norm.
\end{lemma}

The proof of the lemma follows directly from Lemma \ref{lem-exactphija}, together with the estimate \eqref{diff-phija}.  Indeed, let us recall 
$$\phi_{Ray,-}(0) = U_0 - c + \alpha (U_+-U_0)^2 \phi_{2,0}(0) + \cO(\alpha(\alpha + |z_c|))$$
and $\partial_z \phi_{Ray,-}(0) = U_0' + \cO(\alpha \log z_c)$.


\newpage
\section{Construction of fast Orr-Sommerfeld modes $\phi_f$}\label{sec-construction-phi3}


In this section we provide a similar construction to that obtained in Proposition \ref{prop-construction-phi1}. The construction will begin with the fast decaying solution that links with Airy solutions:
\begin{equation}\label{def-phi30}
\phi_{3,0}(z) : = \gamma_0 Ai(2,\delta^{-1}\eta(z)) ,
\end{equation} 
where $\gamma_0 = Ai(2,\delta^{-1}\eta(0))^{-1}$ the normalized constant so that $\phi_{3,0}$ is bounded with $\phi_{3,0}(0) = 1$, $Ai(2,\cdot)$ is the second primitive of the Airy solution $Ai(\cdot)$, and 
\begin{equation}\label{Lg-transform}
\delta = \Bigl( { \eps \over U'_c} \Bigr)^{1/3} , \qquad\quad \eta(z)
= \Big[ \frac 32 \int_{z_c}^z \Big( \frac{U-c}{U'_c}\Big)^{1/2} \; dz \Big]^{2/3}.
\end{equation}
We recall that as $Z$ tends to infinities, $Ai(2, e^{i \pi/6}Z)$ asymptotically behaves as $e^{\mp \frac {\sqrt{2}}{3} |Z|^{3/2}}$. Here $Z = \eta(z)/\delta \approx  (1+z)^{2/3}/\delta$. This shows that $Ai(2, e^{i \pi/6}Z)$ is asymptotically of order $e^{\pm|z/\sqrt{\epsilon}|}$ as expected for fast-decaying modes. 
Consequently, $\phi_{3,0}(z)$ is well-defined for $z\ge 0$ and decays exponentially at $z = \infty$.  Let us recall that the critical layer is centered at $z = z_c$ and has a typical
size of $\delta$. Inside the critical layer, the Airy function plays a role.

\begin{prop}\label{prop-construction-phi3} For $\alpha, \delta$ sufficiently small,  there is an exact solution $\phi_3(z)$ in $X_{\eta/\sqrt\epsilon}$ solving the Orr-Sommerfeld equation
$$ 
Orr(\phi_{3}) = 0
$$
so that $\phi_{3}(z)$ is approximately close to $\phi_{3,0}(z)$ in the sense that 
\begin{equation}\label{est-phi3-Orr} 
 |\phi_{3} (z)  - \phi_{3,0} (z) | \le C \gamma_0 \delta e^{- \eta |z/\sqrt {\epsilon}|},
 \end{equation}
 for some fixed constants $\eta,C$. In particular, at the boundary $z=0$, 
 $$ \phi_3(0) = 1 + \cO(\delta), \qquad \partial_z \phi_3(0) =  \delta^{-1} {Ai(1,\delta^{-1} \eta(0)) 
\over Ai(2,\delta^{-1} \eta(0)) }  (1+\cO(\delta)).$$
\end{prop}

\begin{remark} {\em When the critical layer $z_c$ is away from the boundary, that is, $z_c/\delta$ is sufficiently large, then the solution $\phi_{3,0}(z)$ behaves as an exponential boundary layer. Indeed, since $z$ is near zero, $\eta(z) \sim z-z_c$, $Z = \eta(z)/\delta \sim (z-z_c)/\delta$, and hence we get 
$$ Ai(2,\delta^{-1}\eta(z)) \sim |Z|^{-5/4} e^{\frac {\sqrt{2}}{3} |Z|^{3/2}} \sim |z_c/\delta|^{-5/4} e^{\sqrt{|z_c/\delta|} (z_c - z)/\delta}.$$
Hence, by definition, 
$$ \phi_{3,0}(z) \sim 1 - e^{- \sqrt{|z_c/\delta|} z/\delta},$$
which is exponential.} 
\end{remark}

From the construction, we also obtain the following lemma. 

\begin{lemma}\label{lem-analytic03} 
The fast-decaying mode $\phi_{3}$ constructed in Proposition \ref{prop-construction-phi3} 
depends analytically in $c$ with $\I c \not = 0$. 
\end{lemma}

\begin{proof} This is simply due to the fact that both Airy function 
and the Langer transformation \eqref{Lg-transform} are analytic in their arguments. 
\end{proof}


\subsection{Iterative construction of the Airy mode}

Let us prove Proposition \ref{prop-construction-phi3} in this section.
\begin{proof}[Proof of Proposition \ref{prop-construction-phi3}]
We start with $\phi_{3,0}(z)  =  \gamma_0 Ai(2,\delta^{-1}\eta(z))$. We note that $\phi_{3,0}$ and $\dz\phi_{3,0}$ are both bounded on $z\ge 0$, and so are $\eps \dz^4 \phi_{3,0}$ and $(U-c)\dz^2 \phi_{3,0}$. We shall show indeed that $\phi_{3,0}$ approximately solves the Orr-Sommerfeld equation. In what follows, let us assume that $\gamma_0 =1$. 
Direct calculations yield 
$$\begin{aligned}
 Airy(\phi_{3,0}):= &~\eps \delta^{-1}\eta^{(4)} Ai(1,Z) + 4 \eps \delta^{-2} \eta' \eta^{(3)} Ai(Z) + 3\eps \delta^{-2} (\eta'')^2 Ai(Z)+ \eps \delta^{-4} (\eta')^4 Ai''(Z) \\&+ 6 \eps \delta^{-3} \eta '' (\eta')^2 Ai'(Z)  - (U-c) \Big[\eta'' \delta^{-1}Ai(1,Z) + \delta^{-2}(\eta')^2 Ai(Z)\Big],
\end{aligned}
$$ with $Z = \delta^{-1}\eta(z)$. Let us first look at the leading terms with a factor of $\eps \delta^{-4}$ and of $(U-c)\delta^{-2}$. Using the facts that $\eta' = 1/\dot z$, $\delta ^3= \eps/U_c'$, and $(U-c)\dot z^2 = U_c' \eta(z)$, we have 
$$\begin{aligned}  \eps \delta^{-4} (\eta')^4 Ai''(Z) &- \delta^{-2}(\eta')^2  (U-c)Ai(Z) \\
&=  \eps \delta^{-4} (\eta')^4 \Big[Ai''(Z) - \delta^{2}\eps^{-2}  (U-c)\dot z^2 Ai(Z)\Big]  \\&=  \eps \delta^{-4} (\eta')^4 \Big[Ai''(Z) - Z Ai(Z)\Big] = 0.
\end{aligned}$$
The next terms in $Airy(\phi_{3,0})$ are
$$\begin{aligned}
6 \eps \delta^{-3} \eta '' (\eta')^2 Ai'(Z)  &- (U-c)\eta'' \delta^{-1}Ai(1,Z) \\&= \Big[6 \eta '' (\eta')^2 U_c' Ai'(Z)  - Z U_c' \eta'' (\eta'^2)Ai(1,Z)\Big]
\\&= \eta '' (\eta')^2 U_c'  \Big[ 6Ai'(Z)  - Z Ai(1,Z) \Big] ,\end{aligned}$$ 
 which is bounded for $z\ge 0$. The rest is of order $\cO(\eps^{1/3})$ or smaller. That is, we obtain
 $$ Airy(\phi_{3,0}) = I(z): =  \eta '' (\eta')^2 U_c'  \Big[ 6Ai'(Z)  - Z Ai(1,Z) \Big] + \cO(\eps^{1/3}).$$
Here we note that the right-hand side $I (z)$ is very localized and depends primarily on the fast variable $Z$ as $Ai(\cdot)$ does. Precisely, we have 
\begin{equation}\label{I-bound} |I(z)|\le C (1+z)^{-2} (1+|Z|)^{1/4}e^{-\sqrt 2 |Z|^{3/2}/3}\end{equation}
for some constant $C$. Let us then denote 
$$ \psi(z) : = - AirySolver_\infty (I) (z),$$ 
the exact Airy solver of $I(z)$. It follows that $ Airy(\phi_{3,0} + \psi) = 0$ and there holds the bound
$$
| \psi(z)| \le C\delta (1+z)^{-4/3} (1+|Z|)^{-1/4} e^{-\sqrt 2 |Z|^{3/2}/3}. 
$$
%
%

Next, since $ Airy(\phi_{3,0} + \psi) = 0$, the identity \eqref{key-ids} yields 
\begin{equation}\label{eqs-Orrphi30}
\begin{aligned}
Orr(\phi_{3,0} + \psi) &= I_1(z): = Reg(\phi_{3,0}+\psi) = -  (\eps \alpha^4 + U'' + \alpha^2 (U-c) ) (\phi_{3,0} + \psi).
\end{aligned}
\end{equation}
Clearly, $I_1 \in X_\eta$ for some $\eta \approx 1/\sqrt{\epsilon}$ and $I_1= \cO(\delta)$, upon recalling that $Z  = \eta(z)/\delta \approx (1+z)^{2/3} /\delta$. From this, we can apply the Iter operator constructed previously in Section \ref{sec-construction-phi1} to improve the error estimate. The proposition thus follows. 
\end{proof}

\subsection{First order expansion of $\phi_3$}
By construction in Proposition \ref{prop-construction-phi3}, we obtain the following first order expansion of $\phi_3$ at the boundary 
 $$ \phi_3(0;c) = 1 + \cO(\delta), \qquad \partial_z \phi_3(0;c) =  \delta^{-1} {Ai(1,\delta^{-1} \eta(0)) 
\over Ai(2,\delta^{-1} \eta(0)) }  (1+\cO(\delta)).$$
In the study of the linear dispersion relation, we are interested in the ratio $\partial_z \phi_3 / \phi_3$, on which the above yields 
\begin{equation}\label{ratio-phi3}
 {\phi_{3}(0;c) \over \partial_z \phi_{3}(0;c)} 
 = \delta C_{Ai}(\delta^{-1} \eta(0)) ( 1 + \cO(\delta)),
\qquad \mbox{with}\quad C_{Ai} (Y):= 
 {Ai(2,Y) \over Ai(1,Y) } .
\end{equation}
The following lemma is crucial later on to determine instability. 
\begin{lemma}\label{lem-ratiophi3} Let $\phi_3$ be the Orr-Sommerfeld solution constructed in Proposition \ref{prop-construction-phi3}. There holds
\begin{equation}\label{keyratio-phi3}
 {\phi_{3}(0;c) \over \partial_z \phi_{3}(0;c)} = - e^{\pi i/4} |\delta| |z_c/\delta|^{-1/2} (1+\cO(|z_c/\delta|^{-3/2}))
\end{equation}
as long as $z_c/\delta$ is sufficiently large. In particular, the imaginary part of $\phi_3 / \partial_z\phi_3$ becomes negative when $z_c/\delta$ is large. In addition, when $z_c/\delta = 0$, 
\begin{equation}\label{keyratio-phi3-stable} {\phi_{3}(0;c) \over \partial_z \phi_{3}(0;c)}  = 3^{1/3} \Gamma(4/3) |\delta| e^{5i\pi/6 },
\end{equation}
for $\Gamma(\cdot)$ the usual Gamma function. 
\end{lemma}

Here, we recall that $\delta =  e^{-i \pi / 6} (\alpha R U_c')^{-1/3}$, and from the estimate \eqref{est-eta}, $\eta(0) = - z_c + \cO(z_c^2)$. Therefore, we are interested in the ratio $C_{Ai}(Y)$ for complex $Y = - e^{i \pi /6}y$, for $y$ being in a small neighborhood of $ \RR^+$. Without loss of generality, in what follows, we consider $y \in \RR^+$. Lemma \ref{lem-ratiophi3} follows directly from the following lemma. 

\begin{lemma}\label{lem-CAi} Let $C_{Ai}(\cdot)$ be defined as above. Then, $C_{Ai}(\cdot)$ is uniformly bounded on the ray $Y = e^{7i\pi/6} y$ for $y \in \RR^+$. In addition, there holds 
$$ C_{Ai}(- e^{i \pi /6} y)  =  -  e^{ 5i \pi / 12} y^{-1/2} (1+\cO(y^{-3/2}))  $$
for all large $y\in \RR^+$. At $y = 0$, we have 
$$ C_{Ai} (0) = - 3^{1/3} \Gamma(4/3).$$
\end{lemma}  
\begin{proof} We notice that $Y = - e^{i \pi /6} y$ belongs to the sector $S_1$ defined as in Lemma \ref{lem-expAi12} for $y \in \RR^+$. Thus, Lemma \ref{lem-expAi12} yields 
$$ C_{Ai}(Y)  =  - Y^{-1/2} (1 + \cO(|Y|^{-3/2}))$$
for large $Y$. This proves the estimate for large $y$.  The value at $y = 0$ is easily obtained from those of $Ai(k,0)$ given in Lemma \ref{lem-expAi12}. 
This completes the proof of the lemma. \end{proof}


\newpage
\section{Study of the dispersion relation}\label{sec-disp-relation}



\subsection{Linear dispersion relation}


As mentioned in the Introduction, a solution of \eqref{OS1}--\eqref{OS3} is a linear combination of the slow-decaying solution $\phi_1$ and the fast-decaying solution $\phi_3$. Let us then introduce an exact Orr-Sommerfeld solution of the form
\begin{equation}\label{bl-phiN} \phi: = A \phi_{1} + B \phi_{3},\end{equation}
for some bounded functions $A = A(\alpha,\epsilon,c)$ and $B = B(\alpha,\eps,c)$, where $\phi_1 = \phi_{1}(z;\alpha,\eps,c)$ and $\phi_{3}=\phi_{3}(z;\alpha,\eps,c)$ are constructed in Propositions \ref{prop-construction-phi1} and \ref{prop-construction-phi3}, respectively. It is clear that $\phi(z)$ is an exact solution to the Orr-Sommerfeld equation, and satisfies the boundary condition \eqref{OS3}
 at $z=+\infty$. The boundary condition \eqref{OS2} at $z=0$ then yields the dispersion relations:
$$
\left\{\begin{array}{lrr} \alpha A(\alpha,\eps,c) \phi_{1}(0; \alpha,\eps,c) + \alpha B(\alpha,\eps,c) \phi_{3}(0; \alpha, \eps,c) &=&0\\
A(\alpha,\eps,c) \dz \phi_{1}(0; \alpha,\eps,c) + B(\alpha,\eps,c) \dz \phi_{3}(0; \alpha, \eps,c) &=&0
\end{array}\right.
$$
or equivalently, 
 \begin{equation}\label{dispersion-rel} \frac{\dz \phi_{1}(0; \alpha,\eps,c)}{ \phi_{1}(0; \alpha,\eps,c)} = \frac{\dz \phi_{3}(0; \alpha,\eps,c)}{\phi_{3}(0; \alpha,\eps,c) }.
\end{equation}
We shall show that for some ranges of $(\alpha,\epsilon)$, the dispersion relation yields the existence of unstable eigenvalues $c$.


\subsection{Ranges of $\alpha$}


When $\eps=0$ 
our Orr-Sommerfeld equation simply becomes the Rayleigh equation, 
which was studied in \cite{GGN1} to show that $c(\alpha,0) = U(0) + \cO(\alpha)$ 
and the critical layer $z_c(\alpha,0) \approx \alpha$ (in the case $U'(0)\not =0$; similarly, 
in the case $U'(0)=0$ with possibly a different rate of convergence). Thus, when $\eps>0$, 
we expect that 
$(c(\alpha,\eps), z_c(\alpha,\eps)) \to (U(0),0)$ as $(\alpha,\eps) \to 0$ (which will be proved shortly). 
 In addition, as suggested by physical results (see, e.g., \cite{Reid,Schlichting} or \cite{GGN1} for a summary),
 and as will be proved below, 
  for instability, we would search for $\alpha$ between $(\alpha_\mathrm{low}(R), \alpha_\mathrm{up}(R))$ with 
$$ 
\alpha_\mathrm{low}(R) \approx  R^{-1/4}, \qquad\qquad \alpha_\mathrm{up}(R) \approx R^{-1/6},
$$ 
for sufficiently large $R$. These values of $\alpha_j(R)$ form lower and upper branches of the marginal (in)stability curve for the boundary layer $U$. More precisely, we will show that there is a critical constant $A_{c1}$ so that with $\alpha_\mathrm{low}(R) = A_1 R^{-1/4}$, the imaginary part of $c$ turns from negative (stability) to positive (instability) when the parameter $A_1$ increases across $A_1=A_{c1}$. Similarly, there exists an $A_{c2}$ so that with $\alpha = A_2 R^{-1/6}$, $\I c$ turns from positive to negative as $A_2$ increases across $A_2 = A_{c2}$. In particular, we obtain instability in the intermediate zone: $\alpha \approx R^{-\beta}$ for $1/6<\beta<1/4$. 

We note that the ranges of $\alpha$ restrict the absolute value of $\delta = (\eps/U_c')^{1/3}$ to lie between $\delta_2$ and $\delta_1$, with $\delta_1 \approx \alpha$ and $\delta_2 \approx \alpha^{5/3}$, respectively. Therefore, in the case $\alpha \approx \alpha_\mathrm{low}(R)$, the critical layer is accumulated on the boundary, and thus the fast-decaying mode in the critical layer plays a role of a boundary sublayer; 
in this case, the mentioned Langer transformation plays a crucial role.
 In the latter case when $\alpha \approx \alpha_\mathrm{up}(R)$,
 the critical layer is well-separated from the boundary; in this case, it is sufficient to use
 $\phi_{\mathrm{bl}}$, and we thus replace $\phi_{3}$ by $\phi_{\mathrm{bl}}$ on the right-hand side of our dispersion relation \eqref{dispersion-rel}.

In the next subsections, we shall prove the following proposition, partially confirming the physical results. 

\begin{proposition}\label{prop-Imc} For $R$ sufficiently large, we show that $\alpha_\mathrm{low}(R) = A_1 R^{-1/4}$ is indeed the lower marginal branch for stability and instability. Furthermore, we also obtain instability for intermediate values of $\alpha = A R^{-\beta}$ with $1/6 <\beta <1/4$. In all cases of instability, there holds \begin{equation}\label{bound-Imc}\begin{aligned}
\I c \quad \approx \quad A^{-1}R^{\beta-1/2},
\end{aligned}\end{equation}
and in particular, we obtain the growth rate
\begin{equation}\label{growth} \alpha \I c \quad \approx \quad R^{-1/2}.\end{equation} 
\end{proposition}


\subsection{Expansion of the dispersion relation}

We recall that $U'_0 = U'(0) \not = 0$. By calculations from \eqref{keyratio-phi1} and \eqref{keyratio-phi3}, the linear dispersion relation \eqref{dispersion-rel} simply becomes 
\beq \label{disp2}
\begin{aligned}
 \Big[ U_0 - c + \frac{\alpha (U_+-U_0)^2 }{U_0'} + \cO(\alpha^2\log \alpha ) \Big]  =  \delta C_{Ai}(\delta^{-1} \eta(0)) (1+ \cO(\delta)) 
\end{aligned}\eeq
in which $C_{Ai}(Y) = Ai(2,Y) / Ai(1,Y)$. By Lemma \ref{lem-CAi}, $C_{Ai}(\delta^{-1} \eta(0))$ is uniformly bounded, and asymptotically of order $\cO(|z_c/\delta|^{-1/2})$ for large $z_c/\delta$. In particular, the right hand side of \eqref{disp2} is bounded by $C \delta (1+|z_c/\delta|)^{-1/2}$. As a consequence, 
\begin{equation}\label{cU0-bound} | U_0 - c|  \le C \alpha + C \delta (1+|z_c/\delta|)^{-1/2}.\end{equation}
Hence as $\alpha, \eps, \delta \to 0$, the eigenvalue $c$ converges to $U_0$ and 
\begin{equation}\label{bound-zc}\Big |z_c - \frac{\alpha (U_+-U_0)^2 }{{U_0'}^2}\Big| \le C (\alpha^2\log \alpha + \delta),\end{equation}
followed by the Taylor's expansion: $c = U(z_c) = U_0 + U'_0 z_c + \cO(z_c^2)$.

Next, we give the existence of $c$ for small $\alpha,\epsilon$. 

\begin{lemma}\label{lem-c-existence} For small $\alpha, \epsilon$, there is a unique $c = c(\alpha,\epsilon)$ near $c_0 = U_0$ so that the linear dispersion \eqref{disp2} holds. 
\end{lemma}

\begin{proof} Let us write $c = c_1 + i c_2$ and denote by $F_1,F_2$ the real and imaginary parts of the left-hand side of \eqref{disp2}, respectively. We show that the Jacobian determinant is nonzero at $(\alpha,\eps,c_1,c_2) = (0,0,U_0,0)$. Thanks to Lemmas \ref{lem-analytic03} and \ref{lem-analytic01} with noting that $c_1 = U(z_c)$ so that $\partial_{c_1}z_c = 1/U'_c$ and $\partial_{c_2}z_c = 0$, we can compute 
$$ \partial_{c_1} F_1(\alpha,\eps, c_1,c_2) =-1 +\cO(\alpha)$$ 
Thus, with $(\alpha,\eps) = 0$, $z_c = 0$, and $c_1 = U(0)$, we have 
$ \partial_{c_1} F_1 (0,0,U(0),0) = -1.$
Similarly, we also have 
$$ \partial_{c_2} F_2 (\alpha,\eps,c_1,c_2) =  -1  +\cO(\alpha) ,$$
and thus $ \partial_{c_2} F_2 (0,0,U(0),0) =  - 1 .$ Finally, it is easy to see that $ \partial_{c_2} F_1 (0,0,U(0),0)  =  \partial_{c_1} F_2 (0,0,U(0),0)  =0$. Therefore the Jacobian determinant of $F = (F_1,F_2)$ with respect to $c=(c_1,c_2)$ is equal to one, whereas the Jacobian determinant of the real and imaginary parts of the right-hand side of \eqref{disp2} is of order $\cO(\delta)$ as $\delta \to0$. The standard Implicit Function Theorem can therefore be applied, together with Lemmas \ref{lem-analytic03} and \ref{lem-analytic01}, to conclude the existence of $c = c(\alpha,\eps)$ in a neighborhood of $U_0$.  
\end{proof}


\subsection{Lower stability branch: $\alpha_\mathrm{low} \approx R^{-1/4}$} 


Let us consider the case $\alpha  = A R^{-1/4}$, for some constant $A$. We recall that $\delta \approx (\alpha R)^{-1/3} = A^{-1/3} R^{-1/4}$. That is, $\alpha \approx \delta$ for fixed constant $A$. By a view of \eqref{bound-zc}, we then have $|z_c| \approx C\delta$. More precisely, we have 
\begin{equation}\label{delta-zc} z_c/\delta \quad \approx\quad A^{4/3}.\end{equation}
Thus, we are in the case that the critical layer goes up to the boundary with $z_c/\delta$ staying bounded in the limit $\alpha,\epsilon \to 0$.  

We prove in this section the following lemma.

\begin{lemma}\label{lem-lowerbranch} Let $\alpha  = A R^{-1/4}$. For $R$ sufficiently large, there exists a critical constant $A_{c}$ so that 
the eigenvalue $c = c(\alpha,\epsilon)$ has its imaginary part changing from negative (stability) to positive (instability) as $A$ increases past $A = A_c$. In particular, 
$$
\I c \quad\approx \quad A^{-1} R^{-1/4}.
$$
\end{lemma}
\begin{proof} By taking the imaginary part of the dispersion relation \eqref{disp2} and using the bounds from Lemmas \ref{lem-ratiophi1} and  \ref{lem-ratiophi3}, we obtain 
\begin{equation}\label{disp3}(-1 + \cO(\alpha)) \I c + \cO(\alpha^2 \log \alpha) =\I\Big(  {\phi_{3}(0;c) \over \partial_z \phi_{3}(0;c)} \Big) = \cO(\delta (1+|z_c/\delta|)^{-1/2}).\end{equation}
which clearly yields $\I c = \cO(\delta (1+|z_c/\delta|)^{-1/2})$ and so $\I c \approx A^{-1} R^{-1/4}$. Next, also from Lemma \ref{lem-ratiophi3}, the right-hand side is positive when $z_c/\delta$ is small, and becomes negative when $z_c/\delta \to \infty$.  Consequently, together with \eqref{delta-zc}, there must be a critical number $A_c$ so that for all $A > A_c$, the right-hand side is positive, yielding the lemma as claimed. \end{proof}

\subsection{Intermediate zone: $R^{-1/4} \ll \alpha \ll R^{-1/6}$}


Let us now turn to the intermediate case when 
$$
\alpha = A R^{-\beta}
$$
with $1/10 < \beta <  1/4$.
In this case
$\delta \approx \alpha^{-1/3} R^{-1/3} \approx A^{-1/3} R^{\beta/3 - 1/3}$ and hence
$\delta \ll \alpha$. That is,  the critical layer is away from the boundary: $\delta \ll z_c$ by a view of \eqref{bound-zc}. We prove the following lemma. 

\begin{lemma}\label{lem-midbranch} Let $\alpha  = A R^{-\beta}$ with $1/6<\beta<1/4$. For arbitrary fixed positive $A$, the eigenvalue $c = c(\alpha,\epsilon)$ always has positive imaginary part (instability) with   
$$
\I c \quad \approx\quad   A^{-1 }R^{\beta-1/2}.
$$
\end{lemma}
\begin{proof}
As mentioned above, $z_c/\delta$ is unbounded in this case. Since $z_c \approx \alpha$, we indeed have 
$$z_c/\delta \quad \approx\quad  A^{4/3} R^{(1-4\beta)/3} \to \infty,$$
as $R \to \infty$ since  $\beta <1/4$. By Lemma \ref{lem-ratiophi3}, we then have 
\begin{equation}\label{est-CAi}
\I\Big(  {\phi_{3}(0;c) \over \partial_z \phi_{3}(0;c)} \Big) = \cO(\delta (1+|z_c/\delta|)^{-1/2})  \approx A^{-1} R^{\beta -1/2},
\end{equation}
and furthermore the imaginary of $\phi_3 / \partial_z \phi_3$ is positive since $z_c/\delta \to \infty$. It is crucial to note that in this case 
$$ \alpha^2 \log \alpha \approx R^{-2\beta} \log R,$$
which can be neglected in the dispersion relation \eqref{disp3} as compared to the size of the imaginary part of $\phi_3/\partial_z\phi_3$.

This yields the lemma at once. 
\end{proof}
%
%


\subsection{Upper stability branch: $\alpha_\mathrm{up} \approx R^{-1/6}$}

The upper branch of marginal stability is more delicate to handle.
Roughly speaking, when the expansion of $\phi_{1,\alpha}$ involves $\phi_2$, independent
solution of Rayleigh equation which is singular like $(z - z_c) \log(z - z_c)$. This singularity
is smoothed out by Orr Sommerfeld in the critical layer. This smoothing involves second primitives
of solutions of Airy equation. As we take second primitives, a linear growth is observed 
(linear functions $\phi_R$ are obvious solution of (\ref{Airy})).  This linear growth gives an extra term
in the dispersion relation which can not be neglected when $\alpha \sim R^{-1/6}$. It has a stabilizing
effect and is responsible of the upper branch for marginal stability.

\subsection{Blasius boundary layer: $\alpha_\mathrm{up}\approx R^{-1/10}$} In the case of the classical Blasius boundary layer, we have additional information: $U''(0) = U'''(0) = 0$. Hence, $U''(z_c) = \cO(z_c^2)$, and so by a view of \eqref{defiphi2bis}, the expansion for $\phi_{2,0}$ reduces to 
$$\phi_{2,0} = - \frac{1}{U'_c} + \cO(z_c^2) (U(z) -  c)\log (z - z_c)   + holomorphic .
$$
That is, the singularity $(z-z_c)\log (z-z_c)$ appears at order $\cO(z_c^2)$, instead of order $\cO(1)$ as in the general case. This yields that the singular term $A_s$ that appears in \eqref{def-phi1} is of the form:
 $$A_{s}= \chi  Diff ( e^{-\alpha z}(U-c)) + \chi Diff (\cO(\alpha z_c^2)).$$
Propositions \ref{prop-mAiry} and \ref{prop-exactAiry} thus yield
 $$\| AirySolver_\infty(A_{s})\|_\eta \le C \epsilon + C\alpha^3 \delta(1+|\log \delta|) (1+|z_c/\delta| ),$$
upon recalling that $z_c \approx \alpha$. The dispersion relation \eqref{disp3} then becomes
\begin{equation}\label{disp4}(-1 + \cO(\alpha)) \I c + \cO(\alpha^4 \log \alpha) =\I\Big(  {\phi_{3}(0;c) \over \partial_z \phi_{3}(0;c)} \Big) = \cO(\delta (1+|z_c/\delta|)^{-1/2}).\end{equation}
A simple calculation shows that the right hand side, which has a negative imaginary part, remains to dominate $\cO(\alpha^4 \log \alpha)$ as long as $\alpha \ll \alpha_\mathrm{up} \approx R^{-1/10}$. The fact that $\alpha_\mathrm{up} \approx R^{-1/10}$ is the upper stability branch follows from the same reasoning as discussed in the general case.

\end{document}